\documentclass[psamsfonts]{amsart}

\usepackage{amsmath}
\usepackage{amssymb}
\usepackage{amsthm}
\usepackage{amscd}
\usepackage{xcolor}

\usepackage{verbatim}   
\usepackage{hyperref}

\newtheorem {Theorem}    {Theorem}[section]
\newtheorem*{TheoremA}{Theorem A}
\newtheorem*{TheoremB}{Theorem B}
\newtheorem {Lemma}      [Theorem]    {Lemma}
\newtheorem {Corollary}  [Theorem]    {Corollary}
\newtheorem {Proposition}[Theorem]    {Proposition}
\newtheorem {Claim}{Claim}
\newtheorem {Claim1}{Claim}
\newtheorem*{Conjecture*}   {Conjecture}
\newtheorem{Conjecture} {Conjecture}

\numberwithin{equation}{section}

\newcounter{DM@bibnum}

\newcommand {\N} {{\mathbb N}}
\newcommand {\F} {{\mathbb F}}
\newcommand {\Q} {{\mathbb Q}}
\newcommand {\R} {{\mathbb R}}
\newcommand {\Z} {{\mathbb Z}}
\newcommand {\G} {{\mathbf G}}
\newcommand {\LL} {{\mathcal L}}
\newcommand {\OO} {{\mathcal O}}
\newcommand{\la}{\langle}
\newcommand{\ra}{\rangle}
\newcommand{\nor}{\unlhd}
\DeclareMathOperator{\GL}{GL}
\DeclareMathOperator{\PGL}{PGL}
\DeclareMathOperator{\SL}{SL} 

\DeclareMathOperator{\Aut}{Aut}

\DeclareMathOperator{\Mat}{Mat}
\DeclareMathOperator{\Br}{Br}

\DeclareMathOperator{\Gal}{Gal}
\DeclareMathOperator{\tr}{tr}
\DeclareMathOperator{\End}{End}
\DeclareMathOperator{\rk}{rk}
\DeclareMathOperator{\Irr}{Irr}
\newcommand{\sli}{{\mathfrak{sl}}}

\newcommand{\m}{{\mathbf m}}

\begin{document}

\title{Finite 2-groups with odd number of conjugacy classes}

\author{Andrei Jaikin-Zapirain}
\address{Departamento de Matem\'aticas,
  Universidad Aut\'onoma de Madrid,  and
  Instituto de Ciencias Matem\'aticas - CSIC, UAM, UCM, UC3M.
    28049 Madrid, Spain.  }
\email{andrei.jaikin@uam.es}
\thanks{}

\author{Joan Tent}
\address{Departament d'\`Algebra, Universitat de Val\`encia.
    	46100 Burjassot, Val\`encia, Spain.}
\email{joan.tent@uv.es}

\thanks{This paper is partially supported by the grant MTM 2011-28229-C02-01 and MTM2014-53810-C2-01 of the Spanish MEyC and by the ICMAT Severo Ochoa project SEV-2011-0087. The second author has also been supported by PROMETEOII/2015/011.}


\dedicatory{}

\begin{abstract}
In this paper we consider finite 2-groups with odd number of  real conjugacy classes. On one hand we show that  if $k$ is an odd natural number less than 24, then there are only finitely many finite 2-groups with exactly $k$ real conjugacy classes. On the  other hand we construct infinitely many finite 2-groups with exactly 25 real conjugacy classes. Both resuls are proven using pro-$p$ techniques and, in particular, we use  the Kneser classification of semi-simple $p$-adic algebraic groups.
\end{abstract}

\maketitle

\section{Introduction}
We recall that an element in a group $G$ is  {\bf real} if it is conjugate in $G$ to its inverse,  and a conjugacy class of $G$ is {\bf real} if it
consists of real elements. 
It is well known that certain conditions on 
the set of real conjugacy classes of a finite group may strongly influence the
structure of the group,
the first example of this possibly being that a finite group contains no non-trivial real conjugacy
classes if and only if it has odd order. In the present paper we study finite
$2$-groups satisfying a particular condition of this nature, namely that its number 
of real conjugacy classes is 
an odd natural number.

A theorem by M. Isaacs, G. Navarro and J. Sangroniz \cite{INS} characterizes
the finite $2$-groups of maximal class as those possessing precisely $5$
rational-valued irreducible characters.   It is easy to see that if a $2$-group
has $5$ real irreducible characters, then all of them are rational-valued, and
thus the group has maximal class.   As a consequence of this, 
it is deduced in \cite{INS} that there are (up to isomorphism) only
$3$ finite $2$-groups with $5$ real conjugacy classes, all of them having maximal class. 
Now, a considerably easier fact is that there are no
$2$-groups with $3$ real conjugacy classes.   As a continuation of these results, 
it is proved in \cite{ST} 
that a finite $2$-group with $7$ real classes has order at most $128$.

Indeed, it is a relevant observation that most $2$-groups of small order 
appear to
have an even number of conjugacy classes. 
Since the real classes in a (finite) group are those classes
fixed by the action induced by inversion of elements, 
it is clear that the
parity of the number of conjugacy classes of a finite group 
coincides with the parity of the number of its real classes.
  Then  most small $2$-groups
have an even number of real conjugacy classes.  
In order to explain this phenomenon,
Josu Sangroniz proposed the following conjecture \cite{ST}, which
motivates our work. 

\begin{Conjecture*} If $r$ is an odd natural number, then there are only finitely many finite 2-groups with exactly $r$ real conjugacy classes.
\end{Conjecture*}

  Our  purpose in this paper is to give an answer to this problem. 
More precisely, we shall confirm the conjecture for  $r\le 23$ and show that it is false when $r=25$.

\begin{TheoremA}\label{kless23} Let $r\le 23$ be an odd natural number. Then there are only finitely many finite 2-groups with exactly $r$ real conjugacy classes.
\end{TheoremA}

\begin{TheoremB} \label{k25}There are  infinitely many finite 2-groups with exactly 25 real conjugacy classes.
\end{TheoremB}

In contrast to the previouos results for $2$-groups with an odd number $r\leq 7$ of
real conjugacy classes, which are obtained by methods of finite group theory and its representation theory, 
both Theorem A and Theorem B are proven using pro-$p$ techniques, as we next briefly sketch.

It turns out that the existence of infinitely many finite 2-groups with exactly $k$ real classes is equivalent to the existence of an infinite pro-2 group such that almost all its finite quotients have exactly $k$ real classes. Consequently, the problem is reduced to the study of such pro-2 groups. 
 As we shall
prove,  these groups have finite rank, and so they are 2-adic analytic. 
These key steps
in our  reduction of the problem require to use a basic property of the real classes
in a finite group, which is that its number coincides with the number of real-valued irreducible characters of the group.
 In fact, one advantadge of passing from conjugacy classes to characters
is that the latter behave better with respect to quotient groups, 
and this is exploited in the construction of a pro-$2$
group with $k$ real classes. 
On the other hand, we shall need to look at the real conjugacy classes
again in order to control the rank of this pro-$2$ group. 

 When the number of real conjugacy classes $k$ is odd, we shall show that our analysis should be focused on 
a Sylow pro-2 group of the automorphism group of a semi-simple Lie $\Q_2$-algebra. 
These groups are well understood, because of the classification of semi-simple $p$-adic algebraic groups by  M. Kneser  \cite{Kn}, 
which therefore plays an essential role in the proof of Theorem A. 
It is certainly remarkable that the proof of our result on real classes of finite $2$-groups
needs an appeal to this deep result from the theory of algebraic groups. 
It is ultimately M. Kneser's classification which indicates us where to look for a minimal counterexample to the conjecture, and
as a matter of fact the groups in Theorem B are finite quotients of a Sylow pro-2 group of $\PGL_1(D)$, where $D$ is a division $\Q_2$-algebra of dimension 9 (we recall that up to isomorphism there are 
only two such algebras, both of them having isomorphic multiplicative groups).

 We note that the fact that the equality between the number of real classes and the number of
real characters in a finite group in general has no analogue in the case of rational values,
seems to be an obstruction to apply our arguments in a rather direct way to 
related questions on rational classes and characters in finite $2$-groups.
On the contrary, by this same reason our methods may be valid for similar problems concerning
fields of values in $2$-groups in which equality holds, as it is the case
of  the field $\Q(i)$, being
the fixed field of a cyclic Galois group acting on the classes of a finite $2$-group.

{\bf Acknowledgements.} 
We would like to thank Josu Sangroniz   for useful comments on
preliminary versions of this paper.

\section{Finite group preliminaries}
In this section we include  some preliminary results about real classes of finite groups which we shall need later on. 

Let a group $H$ act on a group $G$  as automorphisms. We will usually consider the action of a subgroup $H$ of $G$ on $G$ by conjugation, and the induced action on quotients of $G$. By a slight abuse of language, we may call the orbits of the action $H$-conjugacy classes.  For a subset $T$ of $G$, we denote by $k_H(T)$ the number of $H$-orbits in $G$ that have non-trivial intersection with $T$.   If $G$ acts on itself by conjugation, then  for simplicity  $k(G)=k_G(G)$  will denote the number of conjugacy classes of $G$.  An element $g\in G$ is called $H$-{\bf real}, or simply  {\bf real}   when there is no possible confusion, if $g$ and $g^{-1}$ are in the same $H$-orbit. A conjugacy class is called { $H$-real} if it consists of $ H$-real elements. We denote by $r_H(T)$ the number of $H$-real conjugacy classes of $G$ that have non-trivial intersection with $T$. For simplicity we   write $r(G)=r_G(G)$.

In the Introduction we already noted that the parities of $k(G)$ and $r(G)$ coincide when $G$ is a finite group. Next we prove this easy fact.  

\begin{Lemma} For a finite group $G$, $k(G)\equiv r(G) \pmod 2$.
\end{Lemma}

\begin{proof}
Observe that inversion of elements in $G$ induces a permutation of order at most $2$ on the set of conjugacy classes of $G$. 
Since a conjugacy class of $G$ is real if and only if it is fixed by this permutation, the result is clear.
\end{proof}

The following results relate the parities of the number of conjugacy classes of a finite 2-group and  a maximal subgroup of it.

\begin{Lemma} \label{congr2}Let $H$ be a subgroup of a finite group $G$ of index 2. Then  $  k(H)\equiv  k_G(G\setminus H) \pmod 2  $.
\end{Lemma}

\begin{proof} Recall that if $G$ acts on a non-empty finite set $\Omega$, then by Burnside's formula  (\cite[Theorem 3.22]{Ro}), the number of orbits
of the action equals to 
$$\frac{1}{|G|}\sum_{g\in G}|\mbox{Fix}_{\Omega}(g)|\, ,$$ where $\mbox{Fix}_{\Omega}(g)$ is the set of fixed points of $g\in G$
in $\Omega$.  Since both $H$ and $G\setminus H$ are normal subsets of $G$ we have that
$$\begin{array}{lll} 2k_G(H) & = & \frac 2{|G|} \left | \{(g,h)\in G\times H: gh=hg\} \right |\\ &&\\ &=&k_H(G)=k_H(G\setminus H)+k(H)=k_G(G\setminus H)+k(H)\, ,\end{array}$$where
the last equality follows from the fact that each $G$-conjugacy class in $G\setminus H$ is already an orbit under conjugation by $H$, 
because $H$ has index $2$ in $G$. 
\end{proof}

\begin{Corollary} \label{subgroupfinite}Let $G$ be a finite group and $H$ a subgroup of index 2. If there are no real elements in $G\setminus H$, then $k(H)$ is even.
\end{Corollary}

If $G$ is a finite group,  we denote   by $\Irr (G)$  the set of irreducible complex characters of $G$. If $\bar G$ is a quotient group of $G$ then $\Irr (\bar G)$ is  identified  in a natural way  with  a subset of 
$\Irr (G)$. We denote by $\Irr _r(G)$ the set of irreducible characters of $G$ taking only real values. Such characters are called {\bf real}.

\begin{Lemma}\label{Brauer} For a finite group $G$, $r(G)=|\Irr _r(G)|$.
\end{Lemma}

\begin{proof}
This is an immediate consequence of Brauer's Theorem 6.32 of \cite{I} (see Problem 6.13 of \cite{I}). 
\end{proof}

As a consequence of   the previous  lemma and the fact that characters of a quotient group
can be identified with characters of the whole group, we obtain the following:

\begin{Corollary} \label{classquot}Let $G$ be a finite group having $r$ real conjugacy classes. Then any quotient group of $G$ has at most $r$ real conjugacy classes.
\end{Corollary}

In the next proposition we compare   the number  of real classes in $H$ and $G\setminus H$ when $H$ has index 2 in a finite group $G$.

\begin{Proposition}\label{bound}
Let $G$ be a finite group and $H$ a subgroup of $G$ of index $2$. Denote by
$\Irr_{r, G}(H)$ the set of real irreducible characters of $H$ which are $G$-invariant.  
Then 
$$
r_G(G\setminus H)\leq |\Irr_{r, G}(H)|\, .
$$In particular,  $r_G(G\setminus H)\leq r(H)$. 
\end{Proposition}

\begin{proof}
Let $\varphi\in\Irr(H)$ be such that there exists $\chi\in\Irr_r(G)$ lying over 
$\varphi$, so $\overline{\varphi}\in\Irr(H)$ is also
a constituent of $\chi_H$.  
Since $G/H$ is cyclic, $\varphi$ is $G$-invariant if and only if 
$\chi_H=\varphi$ by Corollary 11.22 of \cite{I}, so $\varphi$ is real-valued in this case. 
Also,
Gallagher's theorem \cite[Corollary 6.17]{I} implies that if $\varphi$ is $G$-invariant then $\varphi$ has $2$ irreducible
extensions to $G$, both of them being real-valued because $\chi$ is real.  
On the other hand, if $\varphi$ is not $G$-invariant then 
$\chi_H=\varphi + \varphi^a$, where
$a\in G\setminus H$, and $\chi$ is the only irreducible character of 
$G$ lying over $\varphi$, by Clifford's correspondence \cite[Theorem 6.11]{I}. For any $a\in G\setminus H$, this yields to

\begin{multline*}
|\Irr_r(G)|\leq
2 |\Irr_{r,G}(H)| + \frac{|\Irr_r(H)\setminus\Irr_{r,G}(H)|}{2} + \frac{|\{\varphi\in\Irr(H) :\ \varphi^a=\overline{\varphi}\neq\varphi \}|}{2}\\
 = \frac{3|\Irr_{r,G}(H)|}{2}+\frac{|\Irr_r(H)|}{2}+ \frac{|\{\varphi\in\Irr(H) :\ \varphi^a=\overline{\varphi}\neq\varphi \}|}{2}\, .
\end{multline*}
{Since $|\{\varphi\in\Irr(H) :\ \varphi^a=\overline{\varphi}  \}|=|\Irr_{r,G}(H)| +|\{\varphi\in\Irr(H) :\ \varphi^a=\overline{\varphi}\neq\varphi \}|$, we obtain that
\begin{equation}\label{eqchar}
|\Irr_r(G)|\leq
 |\Irr_{r,G}(H)| + \frac{|\Irr_r(H)|}{2} + \frac{|\{\varphi\in\Irr(H) :\ \varphi^a=\overline{\varphi}  \}|}2\, .
\end{equation}}

Now, denote by $r^G(H)$ the number of real conjugacy classes of $H$ which are 
stabilized by $G$ under conjugation. If $K=Cl_H(x)$ for $x\in H$, we write 
$K^{-1}=Cl_H(x^{-1})$ and $K^a=Cl_H(x^a)$, where again $a\in G\setminus H$. For any $a\in G\setminus H$, it is clear that

\begin{multline*}
r_G(H)=r^G(H) + \frac{1}{2}(r(H) - r^G(H)) + \frac{1}{2}|\{ K\in Cl(H) : \ K^a=K^{-1}\neq K\}|\\
= \frac{1}{2}r(H)  + \frac 12 r^G(H) +\frac{1}{2}|\{ K\in Cl(H) : \ K^a=K^{-1}\neq K\}|\, . 
\end{multline*}
{Note that $$|\{ K\in Cl(H) : \ K^a=K^{-1} \}|=r^G(H)+|\{ K\in Cl(H) : \ K^a=K^{-1}\neq K\}|,$$ and so
\begin{equation}\label{eqclas}
r_G(H)=  \frac{r(H)}{2} + \frac{|\{ K\in Cl(H) : \ K^a=K^{-1} \}|}{2}\, . 
\end{equation}}

Let an involution $\sigma$ act on the sets $\Irr(H)$ and $Cl(H)$  
via
$\varphi^\sigma=\overline{\varphi^a}$ and $K^\sigma=(K^a)^{-1}$, respectively, 
for all $\varphi\in\Irr(H)$ and $K\in Cl(H)$. 
By Brauer's Theorem 6.32 of \cite{I}, $\sigma$
fixes the same number of points in $\Irr(H)$ that in $Cl(H)$, that is
$$|\{\varphi\in\Irr(H) :\ \varphi^a=\overline{\varphi}  \}|=|\{ K\in Cl(H) : \ K^a=K^{-1} \}|.$$
Since by Lemma \ref{Brauer}, $|\Irr_r(H)|=r(H)$ and $|\Irr_r(G)|=r(G)=r_G(H) + r_G(G\setminus H)$, it follows from (\ref{eqchar}) and (\ref{eqclas}) that
$$
r_G(G\setminus H) \leq |\Irr_{r,G}(H)|\leq |\Irr_r(H)|=r(H)\, .
$$
\end{proof}

For a finite 2-group $G$, we  consider  the following series of subgroups:
$$G_1=G,\ \ \ G_{k+1}=[G_k,G]G_k^2 \ (k\ge 1).$$
Note that the series eventually reaches the trivial subgroup $\{1_G\}$, because
$G$ is nilpotent. 
For simplicity instead of ``bounded by a function that only depends on $r$" we shall write simply ``$r$-bounded". Recall that the {\bf rank} of a finite group is the supremum of the  number of  generators of its subgroups. (Of course, it is understood that the number of generators of a group $G$ is the size of a minimal generating set of $G$.) 

\begin{Lemma} \label{fgroupsrclases} Let $G$ be a finite 2-group with $r$ real conjugacy classes. Then the following holds:
\begin{enumerate}
\item The number of conjugacy classes of elements of order $2$ in every quotient   group of  $G$ is $r$-bounded.
\item The rank   of $G$ is $r$-bounded.
\item   There exists a $r$-bounded number $k$ such that $\Irr  _r(G)\subseteq \Irr (G/G_k)$.
\end{enumerate}
\end{Lemma}

\begin{proof}
Involutions are real elements, because an element of order $2$ is its own inverse, so  (1)  follows from 
Corollary \ref{classquot}. Now (2)   follows from (1)  and exercise 7  of Chapter 2 in \cite{DDMS}. 
Finally, by  \cite[Theorem C]{ST}, there exists a normal subgroup $N$ of $G$ of $r$-bounded
index in $G$ such that $\Irr _r(G)\subseteq \Irr (G/N)$. Thus  (3)   follows from the nilpotency of $G$. 
\end{proof}

\section{Profinite group preliminaries}
In this section we extend  some of  the results about real characters and real classes from the previous section to   profinite  groups.

If $P$ is a profinite group we denote by $\Irr (P)$ the set of irreducible continuous complex characters of $P$, i.e. $\Irr (P)$ is the union of  $\Irr (\bar P)$ for all the finite continuous quotients $\bar P$ of $P$. As in the finite  order  case, we denote by $\Irr _r(P)$ the subset of characters of $\Irr (P)$ taking only real values. We say that a profinite group $P$ is {\bf $\R$-finite} if the set $\Irr _r(P)$ is finite.
It is not true in general   that   for a profinite group   $P$   the number of conjugacy classes of real elements is equal to the cardinality of $\Irr _r(P)$,  as it is  the case for finite groups. 
For example, in  a non-abelian free pro-2 group    the trivial element forms the unique real conjugacy class,   but the number of real irreducible characters is infinite.

\begin{Lemma}\label{realprofinite} Let $P$ be a profinite group. Then the following holds.
\begin{enumerate}
\item An element $g\in P$ is real if $gN$ is real in $P/N$ for any open normal subgroup $N$ of $P$.\label{guai}
\item  Let $\mathcal S$ be a    closed subset of $P$, $S$  the set of  $P$-real   elements of $\mathcal S$ and $N$ a normal open subgroup of $P$. Then there exists a normal open subgroup $L$ of $P$ contained in $N$ such that the $P/L$-real elements of $\mathcal SL/L$ lie in $SN/L$; 
\item $|\Irr _r(P)|=\sup_{N\nor_o P}|\Irr _r(P/N)|$. 
\item If $P$ is  $\R$-finite, then it has at most $|\Irr _r(P)|$ real conjugacy classes. In particular, every real element of $P$ has finite order in this case.
\end{enumerate}
\end{Lemma}

\begin{proof}
First we prove (1). Suppose that $gN$ is real in $P/N$ for all open normal subgroups $N$ of $P$, 
and define $X_N=\{ x\in P \,|\, (gN)^x=(gN)^{-1} \}\neq \emptyset$ for each $N\nor_o P$.
Note that $\{X_N\}_{N\nor_o P}$ is an inverse system of compact sets with the inclusion maps, 
since $X_N\subseteq X_M$ if $N\subseteq M$ and $N, M\nor_o P$.  Then the inverse limit of this 
inverse set is non-empty (\cite[Proposition 1.1.4]{RZ}), so there exists $x\in P$ such that $gg^x\in N$ 
for all $N\nor_o P$, and thus $g^x=g^{-1}$, as wanted. 

We work by contradiction in order to obtain (2). For each open normal subgroup
$L$ of $P$ contained in $N$, let $X_L$ be the set of elements $x\in \mathcal SL$
such that $xL$ is real in $P/L$ but $x\not\in SN$, and suppose that $X_L\neq \emptyset$
for each such $L$. 
Since the set formed by such subgroups $L$ is a base of neighborhoods of the identity
element, $SN$ is open and $\mathcal S$ is closed,  arguing as in (1) it follows that there 
exists a real element in $\mathcal S$ not
lying in $SN$, which of course is a contradiction. 

Note that (3) follows from the definition of  $\Irr (P)$, 
and  the first part of (4) follows from (3), 
Lemma \ref{Brauer} and the fact that two elements $x, y\in P$ 
are conjugate in $P$ if and only if $xN, yN$ are conjugate for any open normal subgroup
$N$ of $P$.

Finally, suppose that $P$ is $\R$-finite and $x\in P$ is real of infinite order, so 
$M_{x}=\{ o(xN)\, |\, N\nor_{o} P \}$ is infinite. Thus there exist infinitely many $m\in\N$ such that $x\not\in N_m$ and $x^{m}\in N_m$ for some $N_m\nor_{o} P$ depending on $m$, which implies that $x$ is not conjugate to $x^{m}$. Similarly, 
$x^{m}$ is not conjugate to $x^{n}$ if $n,m\in M_x$ with $n\neq m$. Since powers of real elements are real, 
$P$ has infinitely many real classes, which is a contradiction. 
\end{proof}

The following corollary is a consequence of the previous lemma and Corollary  \ref{subgroupfinite}.

\begin{Corollary} 

\label{subgroupprofinite} Let $Q$ be a profinite group and $P$ an open subgroup of index 2. Assume that there are no real elements of $Q$ in $Q\setminus P$. Then there exists an open normal subgroup $L$ of $Q$ contained in $P$ such that for any normal open subgroup $M$ of $Q$ contained in $L$, $r(P/M)\equiv 0\pmod 2$.   Thus   if $ |\Irr _r(P)|$ is finite, then it is even.
\end{Corollary}

\section{$\R$-finite pro-2 groups and pro $2$-groups of finite rank}
In this section we include some basic facts about pro-$p$ groups of finite rank and analyze their  real characters and real classes.

Let $P$ be a profinite group. Denote by $d(P)$ the minimal cardinality of a topological generating set of $P$ (see  \cite[p.\, 20]{DDMS}). Then the {\bf rank} of $P$ is defined to be the supremum of $d(H)$, where $H$ ranges over the closed subgroups of $P$. 
Pro-$p$ groups of finite rank are very well understood, and we refer the reader to \cite{DDMS} for different  equivalent characterizations of this class of groups. We say that a finitely generated pro-$p$ group $N$ is {\bf uniform} if $N$ is torsion-free and $[N,N]\le N^{2p}$. It is a  fundamental result that the pro-$p$ groups of finite rank are exactly the virtually uniform pro-$p$ groups, that is the pro-$p$ groups
having a (normal) uniform subgroup  $U$  of finite index. From this we  immediately obtain that a pro-$p$ group $P$ of finite rank has a unique maximal finite normal subgroup, denoted $rad_f(P)$. In fact, one can obtain a stronger result.

\begin{Lemma}\label{ccs} Let $P$ be a pro-$p$ group of finite rank. Then there are only finitely many conjugacy classes of finite subgroups in $P$.
\end{Lemma}

\begin{proof}
See Theorem 4.23 of \cite{DDMS}.
\end{proof}

Suppose that $P$ is a finitely generated pro-$2$ group. As in the finite order case, consider the following series of subgroups of $P$:
$$
P_1=P,\ \ \ P_{k+1}=[P_k,P]P_k^2 \ (k\ge 1).
$$ If $N$ is a uniform pro-$2$ group, then $N_k=N^{2^{k}}$ and $\{ N_k \}_{k\in\N}$ is a base of open neighbourhoods of
the identity.

In  the   next lemma we collect  some basic properties of $\R$-finite pro-2 groups.

\begin{Lemma} \label{propertiesRfinite}
Let $P$ be a   $\R$-finite pro-2 group. Then the following holds.
\begin{enumerate}
\item 
$P$ has finite  rank;
\item there exists a normal open subgroup $N$ of $P$ such that $N$ is uniform and $\Irr_r(P)\subseteq \Irr(P/N)$;
\item any element of $P$ of  finite order belongs to $rad_f(P)$.\end{enumerate}
\end{Lemma}

\begin{proof}  The first two propositions follow from  Lemma \ref{fgroupsrclases} (2,3).
In order to show (3) it is enough to prove that if $rad_f(P)=\{1\}$, then $P$ is torsion free.
 Thus, assume that $rad_f(P)=\{1\}$ and that there exists  $x\in P$ of order 2. 
   Let $N$ be a uniform normal open subgroup of $P$. 
 Then, since the elements of the form $[x,g]$ with $g\in N$ are   inverted by $x$, we have that they have finite order by Lemma \ref{realprofinite} (4). On the other hand, $[x, N]\leq N$ and $N$ is torsion free, so $x\in C_P(N)$. 
Now since $N$ has finite index in $P$, then $C_P(N)$ is virtually  central. Hence its derived subgroup is finite, by Schur's   theorem  \cite[Theorem 10.1.4]{Rob}. Since $rad_f(C_P(N))\leq rad_f(P)=\{1\}$, we deduce that $C_P(N)$ is abelian. Thus $x\in rad_f(C_P(N))=\{1\}$, a contradiction. 
\end{proof}

Thus, in order to understand $\R$-finite pro-2 groups we have to look at pro-2 groups of finite rank. The most powerful method to study these groups is based on the Lie method that we describe  next.  

Let $L$ be a Lie $\Z_p$-algebra. We say that $L$ is {\bf uniform} if for some $k$, 
$L\cong \Z_p^k$ as  a $\Z_p$-module and
$[L,L]\subseteq 2pL$. 
One can define the functors $\mathbf{exp}$ and $\mathbf{log}$ between the categories of uniform Lie $\Z_p$-algebras and uniform pro-$p$-groups in such a way that these two functors
are   isomorphisms of categories (see \cite[Chapter 4]{DDMS}).  There is a relatively easy way to define the functor $\mathbf {log}$. If $N$ is a uniform pro-$p$ group, then 
  every element $x\in N^{p^i}$ has a unique $p^i$th root in $N$, denoted $x^{1/{p^{i}}}$. Now $\mathbf {log}(N)$ is defined to be the Lie $\Z_p$-algebra whose underlying set coincides with $N$, and which has
Lie operations defined as follows: 
\begin{equation}
\label{lieoperations} a+b=\lim_{i\to \infty} (a^{p^{i}}b^{p^{i}})^{1/p^{i}},\ [a,b]_L=\lim_{i\to \infty}  [a^{p^{i}},b^{p^{i}}]^{1/p^{2i}},\  \forall a,b\in N\, ,
\end{equation} where $[a, b]=a^{-1}b^{-1}ab$  is the commutator defined in the group $N$.  
We write $L_{\Q_p}(N)=\mathbf {log}(N) \otimes_{\Z_p}\Q_p$ and   we refer to this Lie $\Q_p$-algebra as the Lie algebra associated to $N$.

Let $P$ be a pro-$p$ group of finite rank and $N$ an open normal uniform subgroup of $P$. Then   conjugation by $P$ provides $\mathbf {log}(N)$ with the structure of a $P$-module,
via the adjoint representation. 
For $g\in P$ we denote by $ad(g)\in \End_{\Z_p}(\mathbf{log} (N))$  the endomorphism corresponding to the conjugation  action  by $g$:
$$ad(g)(n)=g^{-1}ng,\ \forall n\in N.$$ By a slight abuse of notation, we denote also by $ad(g)$ the extension of $ad(g)$  to $L_{\Q_p}(N)$. The Lie algebra $L_{\Q_p}(N)$, regarded as a $P$-module via  the action of $P$ on it by conjugation, does not depend on the choice of $N$ and it is an invariant of $P$, called the Lie algebra $\LL(P)$ associated to $P$.  If we consider only the Lie structure on $\LL(P)$, then  $\LL(P)$ is a virtual invariant of $P$ (i.e.  the Lie algebra associated to an open subgroup of $P$ is isomorphic to $\LL(P)$).

\begin{Lemma}\label{isom} Let $P$ be a pro-$p$ group of finite rank and $N$ a normal uniform subgroup of $P$. For each $k\ge 0$, we put  $N_k=N^{p^k}$. Let $i,j\in \N$ be such that $i\le j\le 2i+1$. Then $N_i/N_j$ is abelian and 
 $$N_i/N_j\cong \mathbf{log} (N)/p^{j-i}\mathbf {log}(N)$$ as $P$-modules  ($P$ acts on $N_i/N_j$ by conjugation
 and on $\mathbf{log} (N)/p^{j-i}\mathbf {log}(N)$ via the adjoint map).
 \end{Lemma}

\begin{proof}
The lemma is a consequence of the definition of sum in   (\ref{lieoperations}).
\end{proof}

The following lemma provides information about some real elements of pro-2 groups of finite rank.
If $N$ is a finite-index normal subgroup of a profinite group $P$, then $\Irr(P|N)$ denotes
the set of irreducible characters of $G$ whose restriction to $N$ is non-trivial.

\begin{Lemma} \label{eigenvalue} Let $P$ be a pro-2 group of finite rank and $N$ a normal uniform subgroup of $P$.
Then the following statements are equivalent:
\begin{enumerate}
\item there are  $g\in P$ and $0\ne l\in L_{\Q_2}(N)$   such that $ad(g)(l)=- l$;
\item  there are $g\in P$ and $1\ne n\in N$ such that $g^{-1}ng=n^{-1}$;
\item  for infinitely many $t\in \N$  there are $g_t\in P$ and $\psi_t \in \Irr( N_t/N_{2t}|N_{2t-2}/N_{2t})$ such that $\psi_t^{g_t}=\bar \psi_t$.
\end{enumerate}
\end{Lemma}

\begin{proof} (1) and (2) are equivalent statements with different notation
(of course, note that if $-1$ is an eigenvalue of $ad(g)$ in $L_{\Q_p}$, then it is also an eigenvalue
of $ad(g)$ in $\mathbf{log}(N)$).
Suppose (2) holds. Let $t\in \N$. 
The finite group $N_t/N_{2t}$ is abelian of exponent $2^{t}$. If $t$ is large enough, 
by possibly taking a power of $n$ instead of $n$, we can assume that 
$nN_{2t}$ is an element of order $2^{t}$ such that $g^{-1}(nN_{2t})g=(nN_{2t})^{-1}$.
By duality, there exists a character $\psi$ of order $2^{t}$ of $N_t/N_{2t}$ such that $\psi^g=\psi^{-1}=\bar \psi$. 
Since $N_t/N_{2t-2}$ has exponent $2^{t-2}$ we have that $\psi$ is non-trivial on $N_{2t-2}$.

Suppose now  that (3) holds. Note that if  $\psi_t \in \Irr( N_t/N_{2t}|N_{2t-2}/N_{2t})$, then
$\psi_t$ has order at least $2^{t-1}$. Then by duality, for infinitely many
$t\in \N$ there exists $n_t N_{2t}$ in $N_t/N_{2t}$ of order at least
$2^{t-1}$ and $g_t\in P$ such that $g_t^{-1}(n_t N_{2t})g_t=(n_t N_{2t})^{-1}$. 
In particular, note that $n_t\in N_t\setminus N_{t+2}$. By Lemma \ref{isom}, for
any such $t$ there exists $x_t \in \mathbf{log} (N)\setminus 4 \mathbf{log} (N)$ such that
$$ad(g_t)(x_t)\equiv -x_t \pmod{2^t \mathbf{log}(N) }.$$
For any such $t$, let $X_t\subseteq P\times ( \mathbf{log} (N)\setminus 4 \mathbf{log} (N)) $ be
the set of elements $(g,x)$ satisfying $ad(g)(x)\equiv -x \pmod{2^t \mathbf{log}(N) }$, so $X_t$ is non-empty. 
Note that each $X_t$ is closed, and therefore compact. Now $X_{t_1}\subseteq X_{t_2}$
if $t_1 \geq t_2$, and $\{X_t\}_t$ is an inverse system with the inclusion maps. 
The inverse limit of this inverse system is
non-empty (by Theorem 1.4 of \cite{DDMS}), 
so there are $g\in P$ and  $x\in  \mathbf{log} (N)\setminus 4 \mathbf{log} (N)$ such that for infinitely many $t$, $ad(g)(x)\equiv -x \pmod{2^t \mathbf{log}(N) }$. 
Since $\{2^t \mathbf{log}(N)\}$ constitutes a base of open neighbourhoods of the identity element of  $\mathbf{log} (N$), 
we have that $ad(g)(x)=-x$, as wanted.
\end{proof}
 
As a consequence of the above lemma, we obtain that pro-2 groups of finite rank with a finite number of real irreducible characters
admit the following equivalent characterizations.

\begin{Proposition} \label{Rfinite}Let $P$ be a pro-2 group of finite rank. Then the following statements are equivalent:
\begin{enumerate}
\item $P$ is $\R$-finite;
\item any real element of $P$ belongs to $rad_f(P)$;  
\item  $P$ has a finite number of real conjugacy classes;
\item any real element of $P$ has finite order;
\item $-1$ is not an eigenvalue of $ad(g)$ for any $g\in P$. \end{enumerate}
\end{Proposition}

\begin{proof} The fact that (1) implies (3) is already proved in Lemma \ref{realprofinite} (4). 
It is evident that (2) implies (3), and (3) implies (4) by the arguments in
the proof of Lemma \ref{realprofinite}.
If we assume (1) and (4), then (2) follows by Lemma  \ref{propertiesRfinite}(3). Hence we have that (1) implies (2). 
Now (5) follows from (4), because
if $N$ is an open normal uniform subgroup of $P$ 
and $-1$ is an eigenvalue of $ad(g)$ in its action on $L_{\Q_p}(N)$, then by Lemma
\ref{eigenvalue} $N$ contains a non-trivial real element of $P$, which would have infinite order
because $N$ is torsion-free.

Let us prove that (1) follows from (5).
By way of contradiction assume that $P$ has infinitely many real irreducible characters $\{\lambda_i\}_{i\in\N}$. 
As above, suppose that $N$ is an open normal uniform subgroup of $P$. Let $t_i$ be such that $\ker \lambda_i\ge N_{2t_i}$ but $\ker \lambda_i\not \ge N_{2t_i-2}$ and let $\psi_i\in \Irr( N_{t_i}/N_{2t_i}|N_{2t_i-2}/N_{2t_i})$   be an irreducible component of the restriction of $\lambda_i$ on $N_{t_i}$. Since $\lambda_i$ is real, there exists $g_i\in P$ such that $\psi_i^{g_i}=\bar \psi_i$. Hence, by  the implication from (3) to (1) in  Lemma \ref{eigenvalue},  $-1$ is an eigenvalue of $ad(g)$ for some $g\in P$, a contradiction.
\end{proof}

\begin{Corollary} \label{subgroupRfinite}Let $P$ be a $\R$-finite  pro-2 group.  Then any closed subgroup of $P$ is also $\R$-finite.
\end{Corollary}

\begin{proof} Let $T$ be a closed subgroup of $P$. Since $P$ is $\R$-finite, $P$ is of finite rank, by Lemma \ref {propertiesRfinite}(1), and so $T$ is also of finite rank.  A real element of $T$ is also real in $P$. Since all the real elements of $P$ are of finite order, all the real elements of $T$ are of finite order. Thus, the previous proposition implies that $T$ is $\R$-finite.
\end{proof}

In the next lemma we analyze some real elements of quotients  of a semidirect product $P\rtimes \la \phi\ra$, where $\phi$ is an automorphism of $P$ of order 2.

\begin{Lemma}\label{descent}Let $P$ be  a pro-2 group of finite rank, $N$ a normal open uniform subgroup of $P$ and $\phi$ an automorphism of $P$ of order 2 that fixes $N$. Then for any $s> t\ge 2$ and  for any  $n\in N_t$ such that $\phi(n)N_s=n^{-1}N_s$, there exists $m\in N_{t-1}$ such that $nN_{s-1}=\phi(m^{-1})mN_{s-1}$.
\end{Lemma}

\begin{proof} We prove the lemma    by induction on $s-t$. If $s=t+1$ we can take $m=1$. Consider the case $s=t+2$. Note that $N_{t-1}/N_{t+2}$ is abelian. Let  $m\in N_{t-1}$ be such that $m^2=n$. Then $$(\phi(m^{-1})m)^2N_{t+2}=\phi(n^{-1})nN_{t+2}=n^2N_{t+2}.$$ Hence $nN_{t+1}=\phi(m^{-1})mN_{t+1}$.

Now assume $s>t+2$. Then by induction, there exist $l\in N_{t-1}$ such that $nN_{s-2}=\phi(l^{-1})lN_{s-2}$. Let $k=nl^{-1}\phi(l)$. Since $N_{s-2}/N_{s}\subseteq Z(N/N_s)$, $nN_{s}$ commutes with $\phi(l^{-1})lN_s$ and so
$$\phi(k)kN_s=\phi(n)\phi(l^{-1})lnl^{-1}\phi(l)N_s=N_s.$$
Since $k\in N_{s-2}$, by the case $s-t=2$, there exists $v\in N_{s-3}$ such that $kN_{s-1}=\phi(v^{-1})vN_{s-1}$. Hence, since $N_{s-3}/N_{s-1}\le Z(N/N_{s-1})$,
$$nN_{s-1}=k\phi(l^{-1})lN_{s-1}=\phi(v^{-1})v\phi(l^{-1})lN_{s-1}=\phi((lv)^{-1})lvN_{s-1}.$$
\end{proof}

\section{Just infinite quotients of pro-2 groups with odd number of real characters}  
In this section we start the study of  infinite   pro-2 groups having an odd number of real irreducible  characters.   Recall that a profinite group is called {\bf just infinite} if it is infinite and it does not have proper infinite continuous quotients. Every infinite finitely generated pro-$p$ group has at least one just infinite quotient.  
It turns   out that every just infinite quotient   of  a pro-2 group with an odd number of real irreducible characters also has an odd number of real irreducible characters.   

\begin{Theorem}  Let $P$ be an infinite  pro-2 group with a (finite) odd number of real irreducible characters. 
Then $  P/rad_f(P) $ has also an odd number of real irreducible characters.
\end{Theorem}

\begin{proof} Let $N$ be a normal open uniform subgroup of $P$ such that $\Irr_r(P)\subseteq \Irr(P/N)$ (by  Lemma \ref{propertiesRfinite}). Suppose that $K=rad_f(P)$, and write $\bar A=AK/K$ for each $A\subseteq P$
(for simplicity, we denote the images of $N$ and  $N_t$ in $\bar P$ just by $N$ and $N_t$ respectively).
Also, for each $x\in K$, let $T_x=C_P(x)$. Since $N$ and $K$ commute, we have that $N\le T_x$.  
By Lemma \ref{propertiesRfinite}(3), $\bar P$, and so $\bar T_x$, is torsion free.  By Corollary \ref{subgroupRfinite}, $\bar T_x$  is also $\R$-finite, and thus $\bar T_x$ does not
contain non-trivial real elements, by Proposition \ref{Rfinite}(2). 
Again by Proposition \ref{Rfinite}(2) and by Lemma \ref{realprofinite}(2), there exists $k_x\in \N$ 
such that for every $t\ge k_x$,
 the real elements of $P/N_t$ are in $NK/N_t$, and the real elements of $ \bar T_x/   N_t$
lie in $N/  N_t$.

Let $t\geq\max_{x\in K} k_x$. Since $k(P/N_t)$ is odd, there exists $x\in K$ such that $k_P(Cl_P(x)N/N_t)$ is odd. Put $T=T_x$. Hence
$$k_{T}(N/N_t)=k_{T}(xN/N_t)=k_P(Cl_P(x)N/N_t)\, ,$$ since that map sending each $T$-conjugacy class in 
$xN/N_t$ into the $P$-conjugacy class in $Cl_P(x)N/N_t$ which contains it is a bijection, because
$N\cap K=1$. In particular $k_{T}(N/N_t)$
is odd.   Now inversion defines an action on $N/N_t$ of order at most $2$, so $k_T(N/N_t)\equiv r_T(N/N_t)\pmod 2$. 
Thus $\bar T /{N_t}$ has an odd number of real irreducible characters, by the choice of $t$. 
By enlarging $t$ if necessary, we can assume
that $\Irr_r(\bar T)\subseteq \Irr_r(\bar T/ {N_t})$, so $\bar T$ has an odd number of real irreducible characters. Since there are no real elements in $\bar P\setminus \bar T$, Corollary \ref{subgroupprofinite} implies that $\bar P=\bar T$.
\end{proof}

\begin{Corollary} \label{jiq}Let   $P$ be an infinite pro-2 group with odd number of real irreducible characters.   Then any  just infinite quotient of  $P$ has  an odd number of real characters.
\end{Corollary}

\begin{proof}
 Let $N$ be a normal open uniform subgroup of $P$ such that $\Irr_r(P)\subseteq \Irr_r(P/N)$ (Lemma \ref{propertiesRfinite}). Let $M$ be a closed normal subgroup of $P$ such that $P/M$ is just infinite.
Then
$P/(N\cap M)$ has an odd number of real irreducible characters. 
Note that $M/N\cap M = rad_f(P/N\cap M)$, so $P/M$ has also an odd number
of real irreducible characters by the above result. 
\end{proof}

\begin{Corollary}\label{nsol}
Let $P$ be an infinite  pro-2 group with odd  number of real irreducible characters. 
Then  $P$ is not solvable.
\end{Corollary}

\begin{proof} Let $\bar P$ be a just infinite quotient of $P$. Since $\bar P$ is  just infinite, $rad_f( \bar P )=\{1\}$. Now, using that $\bar P$ is $\R$-finite, we obtain that $\bar P$  is   torsion-free. Thus, if $\bar P$ is solvable, it is isomorphic to $\Z_2$ (the only  solvable   torsion-free just infinite pro-2 group). However, $\Z_2$ has only 2 real  irreducible   characters. This is a  contradiction, because $\bar P$ should have an odd number of  real  irreducible  characters by the previous corollary.
\end{proof}
 
\section{Torsion free pro-2 groups with odd number of real irreducible characters}
In this section we show that a  torsion free pro-$2$  group   $P$   with  an   odd number of  real irreducible characters  and such that $P$  can be embedded non-trivially as a subgroup of index 2 in another pro-2 group, contains a proper   pro-2  subgroup with   an odd number of real irreducible characters.  Furthermore, the proper pro-$2$ subgroup of $P$ that we obtain has at most as many real irreducible characters as $P$. This is the main step in the proof of Theorem A  and it will be used in inductive arguments. 

\begin{Theorem} \label{embedding}Let $P$ be a torsion free  subgroup of a pro-2 group $Q$ of index 2. Assume that $P$ has an odd number of real irreducible characters. Then  there exists an element $x\in Q\setminus P$ such that $C_P(x)$  has an odd number of  real irreducible characters and $|\Irr_r(C_P(x))|\le |\Irr_r(P)|$.  
\end{Theorem}

\begin{proof} 
Since $P$ has an odd number of real irreducible
characters and $P$ has index $2$ in $Q$, by Corollary 3.2 there
exists a real element $a$ of $Q$ in $Q\setminus P$. Then $a^2$ lies in
$P$ and it is real in $P$, because $a^2$ is real in $Q$ and it is centralized by $a$. 
Since $P$ is $\R$-finite, Proposition \ref{Rfinite} implies that $a^2$ has finite order, and thus
$a$ has order $2$ because $P$ is torsion free.
In particular, note that every real element
of $Q$ contained in $aP$ is an involution

Observe that 
by Lemma \ref{propertiesRfinite}(1), $P$ has finite rank. Let $N$ be a normal uniform open subgroup of $Q$ contained in $P$.  
Let $S$ be the set of elements of order $2$ in $Q$.
By Lemma \ref {realprofinite}(2) applied to the normal closed subset $\mathcal S=aP$ of $Q$, 
we have that for large $k$ all the real elements of $aP/N_k$ are in $SN_2/N_k$.
Now, by Lemma \ref{ccs}, 
there are only finitely many $Q$-conjugacy classes   in $S$, so let
$\{a_1, \ldots, a_s\}$ be a set of representatives of these conjugacy classes. 

We claim that  
$\{a_1N_2, \ldots, a_sN_2 \}$
lie in different conjugacy classes in $Q/N_2$, and thus the $Q$-orbits of the subsets
$a_1N_2/N_k, \ldots, a_sN_2/N_k $ in $Q/N_k$ are mutually disjoint.  Indeed, 
note that if $(a_iN_2)^g=a_jN_2$ for some $  g\in Q$ and $1\leq i,j\leq s$, then
$n=a_ja_i^g$ lies in $N_2$. Hence $$n^{a_j}n=(a_jn)^2=(a_i^g)^2=1.$$
By Lemma \ref{descent}, for each $s>2$ 
there  exists $m\in N_1$ such that  $nN_{s-1}=(m^{-1})^{a_j}mN_{s-1}$. Hence $$a_i^gN_{s-1}=a_jnN_{s-1}=a_j^mN_{s-1}\, ,$$ and we deduce that $a_j$ is conjugate to $a_i$ in $Q$ because the subgroups $N_s$ constitute 
a base of neighbourhoods of the identity. Thus $i=j$, as wanted. 

By the previous paragraph, we have
\begin{equation} \label{real}
r_{Q}(aP/N_k)=\sum_{i=1}^s r_{Q}(a_iN_2/N_{k}),
\end{equation} for $k$ large enough.  
Since $P$ has index $2$ in $Q$, by  Lemma \ref{congr2} $r_Q(aP/N_k)$ is congruent to $r(P/N_k)$ modulo $2$, for $k\in\N$.
If we take $k$ large enough, then $r(P/N_k)=|\Irr_r(P)|$  by Lemma \ref{propertiesRfinite}(2), so there exists $i$ such that $r_Q(a_iN_2/N_k)$
is odd.

Next we work to prove the main step in the proof, namely that for any element $a\in Q\setminus P$ of order $2$,
\begin{equation}\label{mainid}
|\Irr_r(C_p(a))|=r_Q(aN_2/N_k)\, ,
\end{equation} if $k$ is large enough. Of course, this will complete the proof of the first part of the theorem.   In order to
prove the second part of the result, it is enough to show that
$|\Irr_r (C_P(a))|\leq |\Irr_r (P)|$ for all such $a$. It is clear that (\ref{mainid}) implies that
$|\Irr_r(C_p(a))|\le r_{Q}(aP/N_k)$, and thus 
  if $k$ is 
large enough, then
it follows from Proposition \ref{bound} that $|\Irr_r(C_p(a))|\le |\Irr_r(P)|$. 
Therefore proving (\ref{mainid}) will complete the proof of the theorem.

We fix an element $a\in Q\setminus P$ of order 2 for the rest of the proof. 
By Corollary \ref{subgroupRfinite}, $C_P(a)$ is $\R$-finite.  Let 

$$r=\max\{|\Irr_r(P)|, |\Irr_r(C_P(a))|\}\, .$$ 

The following follows easily from the choice of $r$:

\begin{Claim}\label{realcentralizer}  For $k\geq r $ large enough, suppose that 
$M$ is an open normal
subgroup of $P$ (respectively of $C_P(a)$) contained in $N_k$ (resp. in $C_{N_{k}}(a)$). Then all real
elements of $P/M$ (resp. of $C_P(a)/M$) are contained in $N_{k-r+1}/M$ (resp. in $C_{N_{k-r+1}}(a)/M$).
\end{Claim}

\begin{proof}
By Lemma \ref{Brauer}, $r(P/M)\leq r$.
Since $P$ is $\mathbb R$-finite and
torsion-free, $P$ has no non-trivial real elements
by Proposition \ref{Rfinite}. Then by Lemma \ref{realprofinite} there exists
an open normal subgroup $L$ of $P$   contained in $N$,  such that all real elements
of $P/L$ lie in $N/L$.
Now let $k$ be large enough so that $N_k\leq L$. Suppose that $M\leq N_k$, 
so the real elements of $P/M$ lie in $N/M$. Then  
if $yM$ is a real element of $P/M$ not contained in 
$N_{k-r+1}/M$, the powers of $yM$ provide representatives 
of at least $r+1$ distinct real classes in $P/M$, which is a contradiction.
The same argument works for $C_P(a)$. 
\end{proof}

\begin{Claim}\label{down}Let $k\geq r+2$  be large enough. Then
$r_Q(a N_2/N_k)=r_P(aN_{ k-r}/N_k)\, .$
\end{Claim}

\begin{proof}
Since $P$ has index $2$ in $Q$, it is clear that $Q$-conjugate elements in $(Q\setminus P)N_k/N_k$ are
$P$-conjugate. Thus
$$ 
r_Q(aN_2/N_k)=r_P(aN_2/N_k). 
$$

Let  $y=an\in aN_2$ be such that $yN_k$ is a $P$-real element of $Q/N_k$.  
Then $y^2N_k=n^anN_k$ is real in $P/N_k$. Hence $n^{a}n\in N_{k-r+1}$, by Claim \ref{realcentralizer}.  
By Lemma \ref{descent} ($\phi$ is the conjugation by $a$), there exists $m\in N/N_k$ such that $anN_{k-r}=a^mN_{k-r}$, 
so any $P$-conjugacy class of an element $yN_k\in aN_2/N_k$ which is inverted by $P$ intersects
with $aN_{k-r}$. Therefore: 
$$ 
r_P(aN_2/N_k)=r_P(aN_{k-r}/N_k)\, ,
$$ as wanted.
\end{proof}

Observe that it is clear from the previous claim that if $N_{k-r}\leq H\leq N_2$, then 
\begin{equation}\label{eqsubg}
r_Q(a N_2/N_k)=r_P(aH/N_k)\, .
\end{equation}

\begin{Claim}\label{largel}
 Let $k$ be sufficiently large. Then the following holds: 
\begin{enumerate}
\item The real elements of $C_P(a)/C_{N_{k-1}}(a)$ are in $C_{N}(a)/C_{N_{k-1}}(a)$; \label{centclass}
\item $\Irr_r(C_{P}(a))\subseteq \Irr(C_{P}(a)/C_{N_{k-1}}(a))$;	\label{centch}
\item $C_{P}(aN_{k})\le C_{P}(a)N_{r+2}$;
\item   $  C_{N_k}(aN_{2k})\le C_{N_k}(a)N_{k+r+2}.$
\end{enumerate}
\end{Claim}

\begin{proof}
Since $P$ is $\R$-finite and torsion free, $P$ has no non-trivial real elements. Thus Lemma \ref{realprofinite}(2) (with $P$, $\mathcal S$ and $N$ replaced by $C_P(a)$, $C_P(a)$ and $C_N(a)$, respectively)   implies (\ref{centclass}). Also, Lemma \ref{propertiesRfinite}(2), implies (\ref{centch}).

Now we prove the third statement. 
If there are infinitely many $k\geq r+2$ such that $C_P(aN_k)\not \le C_P(a)N_{r+2}$, then for each such $k$ we can find an element $x_k\in C_P(aN_k)\setminus  C_P(a)N_{r+2}$. Since $P$ is compact, some subsequence $\{x_{k_i}\}$ of $\{x_k\}$ has a limit $x$ that belongs to $\cap_{i=1}^\infty C_P(aN_{k_i})=C_P(a)$. However  $P\setminus C_P(a)N_{r+2}$ is closed, whence $x\not \in C_p(a)\le C_P(a)N_{r+2}$, a contradiction. Thus, there exist only finitely many $k$ such that $C_P(aN_k)\not \le C_P(a)N_{r+2}$. This gives the third statement. 

The last statement can be proved similarly, once one notes that  it is equivalent to the following claim:
$$C_{\mathbf{log} (N)}(a+2^k\mathbf{log} (N))\le C_{\mathbf{log} (N)}(a)+2^{r+2}\mathbf{log} (N).$$
\end{proof}

\begin{Claim}\label{cl:rpq}
Let $k$  be large enough 
and suppose that $N_{k-r-1}\leq H\leq N_{k-r-2}$. Then:
$$r_Q(a N_2/N_k)=r_{C_P(a)N_{r+2}}(aH/N_k)\, .$$
\end{Claim}
 
\begin{proof}
Note that if $ah_1N_k$, $ah_2N_k\in aH/N_k$ are conjugate via $g\in P$, then $(aN_{k-r-2})^g=aN_{k-r-2}$, so
$g\in C_{P}(aN_{k-r-2})$ and thus 

\begin{equation}\label{eq3}
r_P(aH/N_k)=r_{C_{P}(aN_{k-r-2})}(aH/N_k)\, .
\end{equation}
Since $C_P(aN_{k-r-2})\leq C_{P}(a)N_{r+2}$ by Claim \ref{largel}(3),
 we obtain from (\ref{eq3}) that: 
\begin{equation}\label{eq4}
r_{C_{P}(aN_{k-r-2})}(aH/N_k)=r_{C_P(a)N_{r+2}}(aH/N_k)\, .
\end{equation}
Now, if we put together (\ref{eqsubg}), (\ref{eq3}) and (\ref{eq4}), we obtain the claim.
\end{proof}

For $k$  large enough we define $$A=C_{N_{k-r-2}}(a)\, ,$$ and 
$$J=\{z^az^{-1}N_k\in Q/N_k\, :\, z\in N_{r+2} \}\cap N_{k-r-2}/N_k \,.$$

We observe that $J$ is an  abelian subgroup of $Q/N_k$. Indeed, if
${z_i}^{a}{z_i}^{-1}N_k\in J$ with $z_i\in N_{r+2}$ and $i=1,2$, then 
$$z_1^az_1^{-1} z_2^az_2^{-1}N_k=z_1^a z_2^az_2^{-1}z_1^{-1} N_k=(z_1z_2)^a(z_1z_2)^{-1}N_k\, ,$$because 
$N_{k-r-2}/N_k\leq Z(N_{r+2}/N_k)$. So let 
$N_k\leq B$ be the   subgroup of $Q$ such that $B/N_k=J$. Since $C_P(a)$ normalizes $B$, we have that 
$AB$ is a subgroup of $ N_{k-r-2}$, and also $C_P(a)$ acts on the set of
left cosets of $B$ in $AB$.

\begin{Claim}\label{ab} Let $k$  be large enough. Then
$N_{k-r-1}$ is contained in $AB$.
\end{Claim}
\begin{proof} 
Let  $y \in N_{k-r-2}$. Then
$
y^2 = (yy^a)((y^{-1})^{a}y)
$. Note that 
$$yy^a\in C_{N_{k-r-2}}(aN_{2k-2r-4})\, ,$$ and thus $yy^a\in AN_k$ by Claim \ref{largel}(4). Thus $y^2$  lies in $AB$.  Therefore, $N_{k-r-1}/N_k=(N_{k-r-2})^2/N_k\leq AB/N_k$. 
\end{proof}

Note that in particular we have that 

$$
N_{k-r-1}\leq AB\leq N_{k-r-2}.
$$ Thus, by Claim \ref{cl:rpq} if $k$  is large enough then 

\begin{equation}\label{eq5}
r_Q(aN_2/N_k)=r_{C_{P}(a)N_{r+2}}(aAB/N_k).
\end{equation}
 
The aim of introducing the subgroups $A$ and $B$ is to compare the action
by conjugation of $Q$ with that of $C_P(a)$. We have the following: 

\begin{Claim}\label{corresp}
Let $h_1, h_2\in AB$. Then there exists $g\in C_P(a)N_{r+2}$ such that $(ah_1)^gN_k=ah_2N_k$ 
if and only if there exists $c\in C_P(a)$ such that $(h_1)^{c}B= h_2B$. 
\end{Claim}

\begin{proof} Firstly   assume  that $$( h_1)^cB= h_2B\textrm{\ for some  \ } c\in C_P(a).$$ Hence there exists $b\in B$ such that $(h_1)^c= h_2b$. Let $ n\in N_{r+2}$ be such that $ a^nN_k=ab^{-1}N_k$. Hence
$$(ah_1)^{cn}N_k=(a(h_1)^c)^nN_k=a^nh_2bN_k=ah_2N_k\, ,$$where we are using that
$N_{k-r-2}/N_k\leq Z(N_{r+2}/N_k).$
 
Secondly, we assume that  $$(ah_1)^gN_k=ah_2N_k\textrm{\ for some \ } g\in C_P(a)N_{r+2}.$$
Write 
$g=cm$, where $c\in C_P(a)$ and $m\in N_{r+2}$. Hence  $$ah_2N_k=(ah_1N_k)^g=a^mh_1^cN_k,$$ where again the last equality holds because $N_{k-r-2}/N_k$ is centralized by $N_{r+2}$. Thus 
$${aa^mN_k\subseteq B}\, ,$$ and so
$$h_2\equiv h_1^c\pmod {B}. $$  \end{proof}

\begin{Claim}\label{cl:centr} Let $k$  be large enough. Then
$r_{C_P(a)N_{r+2}}(aAB/N_k)=r_{C_P(a)}(AB/B).$
\end{Claim}

\begin{proof} 
Let $n\in AB$. Since $(an)^{-1}=n^{-1}a=an^{-1}[n^{-1},a]$ and $[n^{-1},a]\in B$, by Claim \ref{corresp} we have that $anN_{k}$ is real in $C_P(a)N_{r+2}/N_k$ if and only if there exists $c\in C_P(a)$ such that $n^c\equiv n^{-1} \pmod{B}.$ Thus,
$$r_{C_P(a)N_{r+2}}(aAB/N_k)=r_{C_P(a)}(AB/B).$$ 
\end{proof}

In order to finish the proof, we need one more observation:

\begin{Claim} \label{passtocentr}
Let $k$ be large enough. Then

$$r_{C_P(a)}(AB/B)=|\Irr_r(C_p(a))|\, .$$
\end{Claim}

\begin{proof} 
Observe  that  
$$(z^az^{-1})^a=z (z^{-1})^a= (z^az^{-1})^{-1}\, ,$$ for any $z\in N_{r+2}$.
In particular, $(A\cap B)N_k/N_k$ has exponent 2,
and so $M=A\cap B\subseteq C_{N_{k-1}}(a)$. 
By Claim \ref{realcentralizer}, all real elements of $C_P(a)/M$ are in 

$$C_{N_{k-r}}(a)/M \leq A/A\cap B\cong AB/B.$$ By  Claim \ref{largel}, we have that 
$\Irr_r(C_P(a)) =\Irr_r(C_P(a)/M)$ and the claim follows.
\end{proof}

Now, if we put together (\ref{eq5}), Claim \ref{cl:centr} and the previous claim 
we obtain that

$$
|\Irr_r(C_p(a))|=r_Q(aN_2/N_k)
$$
holds for any $k$   large enough, as wanted. The proof is complete.

\end{proof}

\section{Just infinite pro-2 groups with odd number of real  irreducible characters}
Observe that by Corollary \ref{jiq} and Corollary \ref{nsol}, in order to understand pro-$2$ groups with odd number of real irreducible characters one has to consider non-solvable just infinite pro-$2$ groups of finite rank. Also, 
Theorem  \ref{embedding} indicates that we have to look at ``minimal" examples of   such groups. As   it will follow from the results in this section,   these groups are Sylow pro-2 groups of $\Aut(\sli_1(D))$, where $D$ is some finite-dimensional division algebra over $\Q_2$.

Just infinite pro-$p$ groups of finite rank are well understood.   
We refer to the book \cite{KLP} 
for detailed information  on these groups,   including their classification.   Let $P$ be a non-solvable just infinite  pro-$p$ group of finite rank. Then it is known (see \cite[Proposition III.6]{KLP}) that $\LL(P)$ is a semi-simple Lie $\Q_p$-algebra,    and all summands appearing in the decomposition of $\LL(P)$ as a sum of simple Lie algebras are isomorphic.  Since $P$ is just infinite,   we have that   $\ker ad=\{1\}$. In fact, $ad(P)$ is an open subgroup in $\Aut(\LL(P))$   isomorphic to $P$.  By \cite[Lemma III.16]{KLP}, $\Aut(\LL(P))$ contains a Sylow pro-$p$ subgroup,  that is a maximal pro-$p$ subgroup, and the Sylow pro-$p$ subgroups of $\Aut(\LL(P))$ are all conjugate.

Let $\LL$ be a simple finite-dimensional $\Q_2$-algebra. Then  
the centroid $K$ of $\LL$ is defined as
$$\End_{\LL}(\LL)=\{\phi\in \End_{\Q_2}(\LL): \ [\phi(l),m]=\phi([l,m]) \textrm{\ for every $l,m \in \LL$}\}.$$ It can be shown that $K$ is in fact a finite  extension  field of $\Q_2$, and $\LL$ may be regarded in a natural way as a Lie $K$-algebra, denoted by a slight abuse of notation $\LL$ as well. The $K$-algebra $\LL$ is absolutely simple, i.e. for any field extension $F/K$ the Lie $F$-algebra $\LL \otimes_K F$ is simple (see Chapter X of \cite{Ja}). 

We shall describe a standard construction of an absolutely simple algebraic group 
$\G=\G_{\LL}$ defined over $K$ such that $\G(K)\cong \Aut_K(\LL)$ and moreover, if $E$ is an extension of $K$, then $\G(E)\cong \Aut_E(\LL \otimes_K E)$. This is done in the following way. Fix a basis $B=\{l_1,\ldots,l_n\}$ of $\LL$ over $K$ and define $a^k_{ij}\in K$ as follows
$$[l_i,l_j]_{ L}=\sum^{n}_{k=1} a^k_{ij}l_k.$$
Let  $\phi$ be an automorphism of the Lie $K$-algebra $\LL$  and let $X=(x_{ij})\in \GL_n(K)$ be such that $$\phi(l_i)=\sum_{j=1}^n x_{ij}l_j.$$
Then we obtain that \begin{equation}\label{defeq}\sum_{u=1}^n a_{ij}^ux_{uk}=\sum_{s,t=1}^n a^k_{st}x_{is}x_{jt}, \textrm{\ for all\ }1\le i,j,k\le n.\end{equation}
Then we define  $\G=\G_{\LL}$ to be the algebraic subvariety of $\mathbf{GL}_n$ defined by the equations (\ref{defeq}).  In fact, one can easily check that $\G$ is an algebraic subgroup of $\mathbf{GL}_n$ and $\G(E)\cong \Aut_E(\LL \otimes_K E )$. 

Denote by $\G^o$ the connected component of  the identity of $\G$. Note that  $\G^o$ is also defined over $K$ (see \cite[34.2]{Hu}). 
The algebraic group $\G^o$ is of adjoin type (see \cite[31.1]{Hu}). Denote  by $ \tilde {\G}^o$ the simply connected cover of $\G^o$ (see \cite[Theorem 2.6]{PR}). By \cite[Proposition 2.10]{PR}, there exists a universal covering $\pi :\tilde {\G}^o\to {\G}^o$ defined over  $K$. In particular $\tilde {\G}^o$ can be defined over $K$.

Since $\G$ is a $K$-group, its Lie algebra (see \cite[9.1]{Hu}) has a $K$-structure (see \cite[34.2]{Hu}). Arguing as in the proof of   \cite[Corollary 13.2 ]{Hu}, we obtain that the $K$-points of the Lie algebra of $\G$ is isomorphic to the algebra $\mathfrak{Der}_K(\LL)$ of $K$-derivations of $\LL$. Since $\LL$ is simple,  $\mathfrak{Der}_K(\LL)\simeq \LL$ (see \cite[14.1]{Hu}). Thus, the $K$-points of the Lie algebra of $\G$ are isomorphic to $\LL$.   In particular, $\G$ is absolutely simple.  

Now, we consider  a particular case of the previous situation.
If $K$ is a field, 
we say that a division algebra $D$ is  $K$-central if $Z(D)\cong K$ and $D$ has 
finite dimension over $K$.
Suppose that $K$ is a finite extension of $\Q_2$, 
and
let $D$ be a finite-dimensional $K$-central division algebra. Suppose that $E$ is a splitting field for $D$, so by definition there exists an isomorphism of $E$-algebras $\varphi: D\otimes_{K} E \simeq M_d(E)$, where $d$ is the index of $D$ (recall that $E$ can be taken to be any field containing a maximal subfield of $D$, see Theorem 7.15 of \cite{Rei} and the discussion before it). For each element $x\in D$ we denote its {\bf reduced trace} by way of 
$$\mbox{trd}(x)= \tr (\varphi(x\otimes_{K} 1_E))\, .$$
Let  $\sli_1(D)$ be the elements of $D$ of reduced trace 0. Then $\sli_1(D)$ is a simple Lie $\Q_2$-algebra with Lie bracket defined by $[x,y]_L=xy-yx$.   One can show that in fact the algebraic group $\G_{\sli_1(D)}$ is already connected. Let us describe its simply connected cover.

The {\bf reduced norm} is the map given by $\mbox{Nrd}(x)=\det (\varphi (x\otimes_{K} 1))$, for all $x\in D$. It follows from the definition that $\mbox{Nrd}$ is a multiplicative map. Also, the reduced norm is independent of the choice of $E$ and $\varphi$,
and it takes values in $K$ (see Theorem 9.3 and p. 116 of \cite{Rei}).

It is well-known that $\mbox{Nrd}(x)$ is given by a homogeneous polynomial of degree $d$ with coefficients in $K$, in the coordinates of $x\in D$ with respect to an arbitrary fixed basis for $D$ over $K$ (see p. 27 of \cite{PR}). Next we define, following 2.3.1 in \cite{PR}, an algebraic $K$-group $   \mathbf {SL}_1 (D)$ whose group of $K$-rational points is $\SL_1(D)=\{ x\in D^*\, :\, \mbox{Nrd}(x)=1\}$. For $x\in D$, let $r_x$ be the $K$-linear map in $D$ induced by right mupliplication by $x$. The algebra representation $\rho:D\rightarrow M_{d^2}(K)$ sending each $x\in D$ to the matrix $\rho(x)$ of $r_x$ with respect to a fixed $K$-basis for $D$ is called the {\bf regular representation} of $D$. 
Observe that the $K$-linear subspace $\rho(D)$ of $M_{d^2}(K)$ is the set of common zeros of a finite number of linear polynomials $f_1,\ldots,f_t$ with coefficients in $K$, in the coordinates of the matrices with respect to the canonical basis of $M_{d^2}(K)$. There exists a polynomial $f$ with coefficients in $K$ such that if $\rho(x)=\big{(}k_{ij}\big{)}\in M_{d^2}(K)$, then
$$f(k_{11}, \ldots, k_{1d^2}, k_{21}, \ldots,k_{d^2d^2})=\mbox{Nrd} (x)\, ,$$ for all $x\in D$. { We define  $\mathbf{SL}_1 (D)$ to be the algebraic subvariety of $\mathbf{GL}_{d^2}$ defined by the equations $f_1=\ldots=f_t=f-1=0$. }

Now let $\Omega$ be an algebraic closure of $K$ containing $E$. Observe that the vanishing set of $f_1, \ldots, f_t$ in $M_{d^2}(\Omega)$ is isomorphic as an $\Omega$-algebra to $D\otimes_K \Omega\cong M_d(\Omega)$. Thus, if $b=\big{(} b_{ij}\big{)}\in M_{d^2}(\Omega)$ annihilates all polynomials $f-1,f_1, \ldots, f_t$, then $b$ may be regarded as a point in $M_d(\Omega)$, and as such it has
determinant equal to $1$, { because $f(b)=1$.
It then follows that $ \mathbf   {SL}_1(D)$ is an algebraic $K$-group $\Omega$-isomorphic to $\mathbf{SL}_d$, and in particular it is simply connected. 

Now, $\pi:\mathbf{SL}_1(D)\to \G_{\sli_1(D)}$ is defined in an obvious way because any point of $D\otimes_K \Omega$ induces by conjugation an automorphism of $\sli_1(D)\otimes_K \Omega$.}

\begin{Theorem}\label{sylow} Let $\LL$ be a finite-dimensional  semi-simple  Lie $\Q_2$-algebra
all whose simple components are isomorphic. 
Assume that the Sylow pro-2 subgroups of $\Aut(\LL)$ are torsion free. Then $\LL\cong \sli_1(D)$ for some finite-dimensional division $\Q_2$-algebra $D$.  
\end{Theorem}  

\begin{proof}
First observe that $\LL$ is simple. Indeed, if $\LL$ decomposes as a direct sum of more than one simple algebra, then the automorphism of $\LL$ that permutes two of the summands into which $\LL$ decomposes and acts as the identity on the rest of summands has order 2, which is a contradiction.

We use the notation introduced before the theorem. Let $\G=\G_{\LL}$. Hence $\G(K)\cong \Aut_K(\LL)$.  First, let us show that $\G^o$ is $K$-isotropic, i.e. that the $K$-rank of $\mathbf  G^o$ is zero \ (see \cite[34.5]{Hu} for   a definition of $K$-rank). Assume that  the  $K$-rank of $\G^o$ is not zero.  Then $\G^o$ contains a $K$-split torus $T$ of positive dimension $n$. Since the group $T(K)$ of $K$-points of $T$, which  is isomorphic to  $(K^*)^n$,  contains an element of order 2 (because  -1 is an element of $K^*$ of order 2), we obtain a contradiction.
 
Since $\G^o$ is $K$-isotropic,  $\tilde {\G}^o$ is also $K$-isotropic. Hence, by \cite{Kn} (see also \cite[Theorem 6.5]{PR}), $\tilde{\G}^o\cong \mathbf{SL}_1  (D)$ for some $K$-central division algebra $D$.   Thus, $ \LL $ is isomorphic to $\sli_1(D)$.  \end{proof}

\section{Real characters of Sylow 2-pro groups of $\PGL_1(D)$}

\setcounter{Claim}{0}
In this section we study the Sylow  pro-2 subgroups of $\Aut(\sli_1(D))$ having odd number of real irreducible characters,  
where $D$ is a $K$-central division algebra over a finite degree extension $K$ of $\Q_2$.

Let $K$ be a field. It is well-known (see \cite{Se}) that the isomorphism types of $K$-central division algebras are classified by the elements of the  group $H^2(\Gal(K_s/K),K_s^*)$, where $K_s$ denotes the maximal algebraic separable extension of $K$. The situation when  $K$ is a finite extension of $\Q_p$ is perfectly understood. From local class field theory we know that $H^2(\Gal(K_s/K),K_s^*)=\Q/\Z$  in this case. In particular,   this implies  that there are exactly $\phi(d)$ different (up to isomorphism) $K$-central division algebras of dimension $d^2$ over $K$ and one can easily describe these algebras. We recommend the reader to look at \cite{Se} for details and proofs of the results that we present in the next paragraphs.

Let   $K$ be a finite extension of $\Q_2$.
Next we describe how to construct the different $K$-central division algebras of a given dimension $d^2$. 
Denote by $\OO_K$ the ring of integers of $K$, which is a local ring, and let $\m_K$ be its maximal ideal. Then $\m_K$ is principal, and we choose a generator $\pi_K$ of $\m_K$. The residue class field $\OO_K/\m_K$ has order a power of $2$, say $q$. Now, let $w$ be a $(q^d-1)$th primitive root  of unity and write $W=K(w)$. Then $W/K$ is an unramified, cyclic Galois extension of degree $d$. The Galois group  $\Gal(W/K)$ is generated by the $K$-automorphism $\theta$ sending $w$ to $w^q$, so the $\phi(d)$ distinct powers $\theta^r$ of $\theta$, where $r$ is coprime to $d$, are the generators of $\Gal(W/K)$. Choose such a generator $\alpha=\theta^r$ sending $w$ to $w^{q^r}$.  Then the $K$-algebra $D$ generated by $W$ and an element $\pi_D$ satisfying the following relations:
\begin{equation}\label{relations}
(\pi_D)^d=\pi_K \textrm{\ and\ } (\pi_D)^{-1} v \pi_D= {\alpha}(v) \textrm{\ for all $v\in W$}
\end{equation} is a $K$-central division algebra. 
The set $\{1,\pi_D,\ldots,(\pi_D)^{d-1}\}$ is a basis of $D$ over $W$, so $D$ has dimension $d^2$ over $K$.

The following proposition   describes   the structure of $\Aut_{\Q_2}(\sli_1(D))$.

\begin{Proposition}\label{autsl1new} Let $K$ be a finite extension of $\Q_2$, $D$ a $K$-central division algebra of dimension $\dim_K D=d^2>1$ and  $\LL=\sli_1(D)$. Let   $\Psi: \Aut_{\Q_2}(D)\to \Aut_{\Q_2}(\LL)$ be the restriction map. Then  the  map $\Psi$ is an isomorphism.
\end{Proposition}

\begin{proof}
 It is clear that any $\Q_2$-automorphism of $D$ fixes the Lie algebra $\LL$  and  it induces   an automorphism of $\LL$. Thus,
 $\Psi$ is well-defined. It is also clear that the map $\Psi$  is a monomorphism. Let us show that $\Psi$ is surjective. We divide the proof of    this into   several steps. 

\begin{Claim1}  Let $\phi$ be an element of $ \Aut_{\Q_2}(\LL)$. Then $\phi$ induces an automorphism  of $K$ (that we will denote also by $\phi$) such that 
$$\phi(k)\phi(l)=\phi(kl) \textrm{\ for every \ } k\in K \textrm{\ and \ } l\in\LL. $$
\end{Claim1}
\begin{proof} 
For every $k\in K$, let $\alpha_k:\LL\to \LL$ be a $\Q_2$-linear map defined by means of
$$\alpha_k(l)=\phi (k\phi^{-1}(l))\ (l\in \LL).$$ Then we  see that for any $l,m\in \LL$ we have that
\begin{multline*}[\alpha_k(l),m]_L=[\phi (k\phi^{-1}(l)), m]_L=\phi([k\phi^{-1}(l),\phi^{-1}(m)]_L=\\
\phi(k\phi^{-1}([l,m]_L)=\alpha_k([l,m]_L).\end{multline*}
Hence $\alpha_k$  is an element of the centroid of $\LL$, which is equal to $K$ 
(since $\mathbf G_\LL$ is an absolutely simple $K$-group, this follows from \cite[Theorem 3 of Chapter X]{Ja} by construction of $\mathbf G_\LL$). Thus, $\alpha_k$ is represented by multiplication by a (uniquely defined)  element of $K$.  We denote this element of $K$ by $\phi(k)$. One easily checks that $\phi:K\to K$ is an automorphism of $K$.
\end{proof}

 \begin{Claim1}\label{extenv}  Let $\phi$ be an element of $ \Aut_{\Q_2}(\LL)$. Then $\phi$ induces   a $\Q_2$-algebra automorphism of the enveloping $K$-algebra $U_K(\LL)$ of $\LL$ (that we will denote also by $\phi$)  
that extends the action of $\phi$ on $K$ and $\LL$.
 
 \end{Claim1}
 \begin{proof} Note that $U_K(\LL)$ can be defined as a $\Q_2$-algebra generated by two $\Q_2$-vector spaces $K$ and $\LL$ with the following relations:
$$k\cdot l=l\cdot k=kl,\ l\cdot m-m\cdot l=[l,m]_L\ (k\in K, l,m\in \LL).$$
In the previous relations $\cdot$ is the multiplication in the algebra $U_K(\LL)$ and $kl$ is considered as an element of   the space $\LL$. Since $\phi$ conserves these relations we can extend $\phi$ on $U_K(\LL)$. 
 \end{proof}

Now, we will use an argument that appears in Lemma \cite[XI.14]{KLP}.
Note that $\bar K\otimes_K D\cong \Mat_d(\bar K)$ and $\bar K\otimes_K U_K(\LL)\cong U_{\bar K}(\sli_d(\bar K))$.
\begin{Claim1}\label{extfield} Let $I_1$ and $I_2$ be  two ideals of  $U_K(\LL)$. Then $I_1\le I_2$ if and only if $\bar KI_1\le \bar KI_2$.
\end{Claim1}
\begin{proof} We have to prove only  the ``if" part. For this, observe that if $\bar KI_1\le \bar KI_2$, then
$$I_1\le U_K(\LL)\cap \bar KI_2=I_2.$$
\end{proof}
The embedding of $\LL$ into $D$  induces a surjective homomorphism $U_K(\LL)\to D$. We denote by $I$ the kernel of this homomorphism.
The Lie algebra $\LL$ can be also embedded into $D^{op}$ (the opposite algebra of $D$) by sending $l\in \LL$ to $-l\in D$. This   embedding   induces a surjective homomorphism $U_K(\LL)\to D^{op}$. We denote by $J$ the kernel of this homomorphism. 
\begin{Claim1}\label{dopd} Let $d\ne 2$.  Then $D\not \cong  D^{op}$ as $\Q_2$-algebras. In particular, $I$ and $J$ are not equal.
\end{Claim1}
\begin{proof} We have described the structure of $D$ before this proposition,  and we use the notation introduced there. The relevant parameters that determine $D$ up to  $K$-isomorphism are its center $K$, its index $d$ and the number $r$ ($1\le r\le d$, $(r,d)=1$). Another division $K$-algebra, with the same parameters as $D$, is $K$-isomorphic to $D$.  Let us recall briefly how one can define $r$  internally in terms of $D$. 

The division algebra $D$ contains a unique maximal compact subring $\OO_D=\OO_W[\pi_D]$ and $\pi_D\OO_D$ is its maximal ideal. Recall that $q$ is the size of the field $\OO_K/\m_K$ (and so, depends only on $K$). The group  $D^*$ contains  a unique maximal compact subgroup $\OO_{D^*}$ which is known to be prosoluble of order $(q^d-1)2^{\infty}$. Hence $D^*$ contains a unique  conjugacy class  of   subgroups of order $q^d-1$ (which are, in fact, cyclic). Let $A=\langle a\rangle $ be one   of   such subgroups (for example, it can be $\langle w\rangle$).    Take  $g\in N_{D^{*}}(A)\cap ( \pi_D\OO_D\setminus \pi_D^2\OO_D)$. Then $a^g=a^{q^r}$ and $r$ does not depend on the choices of $a$ and $g$. 

Let $\psi:D\to D^{op}$ be an isomorphism of $\Q_2$-algebras. We want to show that $D$ and $D^{op}$ are isomorphic as $K$-algebras as well. The center of $D^{op}$ is $K$ and  its index is equal to $d$. So we only have to analyze the parameter $r$ corresponding to $D^{op}$.

Note that $\psi$ is   a topological isomorphism  as well.  Hence,  $\phi(\OO_D)=\OO_{D^{op}}$ and $\psi(\pi_D\OO_D)=\psi(\pi_D)\OO_{D^{op}}$ is the maximal ideal of $\OO_{D^{op}}$.  Put $A=\langle \psi(w)\rangle $ and $g=\psi(\pi_D)$. Since $a^g=a^{q^r}$, we obtain from the previous discussion that   $D^{op}$ and $D$   are isomorphic as $K$-algebras.

Let $[D]\in \Br(K)$ denote the element of  the Brauer group of $K$   corresponding to the algebra $D$. Since $K$ is a   local field, the order of $[D]$ is $d$. Note that $[D]^{-1}=[D^{op}]$ and so if $D\cong D^{op}$, $d=2$,   a contradiction. 
 \end{proof}
 \begin{Claim1} Let $Z$ be an ideal of $K$-algebra $U_K(\LL)$ of codimension $d^2$ such that $U_K(\LL)/Z$ is a central  simple $K$-algebra. Then  if $d=2$, $Z=I$ and if $d>2$ $Z=I$ or $Z=J$.
\end{Claim1} 
\begin{proof}  Note that $\sli_d(\bar K)$ has a single irreducible module (up to isomorphism) of dimension 2 if $d=2$ and exactly two irreducible modules of dimension $d$ if $d>2.$ 

Thus, if $d=2$, $\bar KI$ is the only ideal of $U_{\bar K}(\sli_d(\bar K))$ such that the quotient of $U_{\bar K}(\sli_d(\bar K))$ by this ideal is isomorphic to $\Mat_d(\bar K)$. Claim \ref{extfield} implies that $I=J$ is the only  ideal of $U_{K}(\LL)$ such that the quotient of $U_{ K}(\LL)$ by this ideal is a central  simple $K$-algebra of dimension $d^2$.

 Assume now that $d>2$. Since $I\ne J$, by Claim \ref{extfield}, $\bar KI\ne \bar KJ$. Hence, if $d>2$ there are exactly two   ideals $\bar KI$ and $\bar KJ$ of $U_{\bar K}(\sli_d(\bar K))$ such that the quotient of $U_{\bar K}(\sli_d(\bar K))$ by this ideal is isomorphic to $\Mat_d(\bar K)$, and so,  there are exactly two   ideals $I$ and  $J$ of $U_{K}(\LL)$ such that the quotient of $U_{K}(\LL)$ by these ideals is  a central  simple $K$-algebra of dimension $d^2$.
 \end{proof}

\begin{Claim1} \label{imI}Let $\phi$ be an element of $ \Aut_{\Q_2}(U_K(\LL))$. Then $\phi(I)=I$.
\end{Claim1}
\begin{proof} Note that $U_K(\LL)/\phi(I)$ is   isomorphic to $D$ (as $\Q_2$-algebra). Thus, $\phi(I)$ is an ideal of  $K$-algebra $U_K(\LL)$ of codimension $d^2$ and  $U_K(\LL)/\phi(I)$ is a central  simple $K$-algebra. Hence by the previous claim,  $\phi(I)=I$ or $  \phi(I)=  J$. In the first case we are done.  

Now, assume that $d>2$ and   $\phi(I)= J$. We want to show that it cannot happen. If this happens, then    $\phi$ induces an $\Q_2$-isomorphism between  $U_K(\LL)/I\cong D$ and $U_K(\LL)/J\cong D^{op}$, but this is impossible by Claim \ref{dopd}.\end{proof}

Now, we are ready to finish the proof of the proposition. Let $\phi$ be an element of $ \Aut_{\Q_2}(\LL)$. Then by Claim \ref{extenv}, $\phi$ extends to a  $\Q_2$-automprhism of $U_K(\LL)$. By Claim \ref{imI}, $\phi(I)=I$.   Hence $\phi$ induces a  $\Q_2$-automorphism of  $U_K(\LL)/I\cong D$ which extends the action of $\phi$ on $\LL$. This finishes the proof of  the proposition.

  \end{proof}
  We write $D^*$ for the multiplicative group of non-zero elements of $D$ and $\PGL_1(D)=D^*/Z(D^*)$. Recall that by the Skolem-Noether Theorem, $\Aut_K(D)\cong \PGL_1(D)$.
 
Now let us describe the structure of $\PGL_1(D)$.
Let $\OO_D=\OO_W[\pi_D]$ and let $U_D$ be its unit group. 
Note that  $\m_D=\OO_D \pi_D=\pi_D\OO_D=(\pi_D)$ is the unique   maximal ideal of $\OO_D$, and $\OO_D/\m_D\cong \OO_W/\m_W \cong\F_{q^d}$. 
Define the following subgroups of  $\PGL_1(D)$ : let \  $C=\la \pi_{ D}\ra Z(D^*)/Z(D^*)$, $U_0=U_DZ(D^*)/Z(D^*)$ and  for $i\ge 1$, $U_i=(1+\m_D^i)Z(D^*)/Z(D^*)$.

\begin{Proposition} \label{pgl1d}   The following holds:
\begin{enumerate}
\item $C$ is a cyclic group of order $d$.\label{skewfield_first}
\item $U_0$ is a normal subgroup of $\PGL_1(D)$ and $\PGL_1(D)$ is a semidirect product of $U_0$ by $C$.\label{skewfield_second}
\item For  $i, j\ge 1$ , $U_i$  is  a normal pro-2 subgroup of $\PGL_1(D)$ and $[U_i,U_j]\le U_{i+j}$. Moreover,
$$U_i/U_{i+1}\cong\left \{ \begin{array} {ll}
 \F_{q^d} & \textrm{if $i\not \equiv 0\mod d$}\\
\F_{q^d}/\F_q & \textrm{if $i \equiv 0 \mod d$}\end{array}\right . , 
$$where $\F_{q^d}=(\F_{q^d}, +)$ denotes the additive group of the field. \label{skewfield_third}
\item $U_0/U_1$ is isomorphic to the multiplicative group
of $\F_{q^d}$. \label{skewfield_new}

\item Let $q=2^r$ and $n=\dim _{\Q_2}K$.  If $i > \frac{dn}{r}$ then
$U_{i}^2= U_{i+\frac{dn}{r}}$.
\label{skewfield_fourth}\end{enumerate}

\end{Proposition}

\begin{proof}
It is clear that (\ref{skewfield_first}) follows at once from the relations (\ref{relations}). 

Since $\OO_D=\OO_K[w, \pi_D]$
and $w\in U_0$, we have that $\PGL_1(D)=U_0C$. It is clear that $U_0$ is a normal subgroup of $\PGL_1(D)$ which
intersects $C$ trivially (because any power of $\pi_D$ 
in $U_DZ(D^*)$ is already in $Z(D^*)$), so (\ref{skewfield_second}) follows. 

The normality of the subgroups $U_i$ 
and the statement on the commutator subgroups are straightforward. By \cite[Proposition 1.8]{PR}, the map 
$1+a\pi_D^{i}\mapsto a+\m_D$ induces an isomorphism from $(1+\m_D^i)/(1+\m_D^{i+1})$
into the abelian group of the residue field $\OO_D/\m_D\cong\F_{q^d}$. It follows from (\ref{skewfield_first})
that moding out by $Z(D^*)$ only affects the quotient group $(1+\m_D^i)/(1+\m_D^{i+1})$ 
when $i\equiv 0\mod d$,
in which case $U_i/U_{i+1}\cong \F_{q^d}/\F_q$ because $(1+\OO_D \pi_D^{i})\cap Z(D^*)=1+\OO_K   \pi_D^{i}$
and $\OO_K\, \cap\, \m_D   = \m_K$ (see Theorem 13.2 of \cite{Rei}). 

(\ref{skewfield_new}) follows from  Proposition 1.8 of \cite{PR}.   

It remains to prove (\ref{skewfield_fourth}). The ramification index of $K/\Q_2$
is $n/r$ by \cite[Theorem 13.3]{Rei}, so we have that there exists a unit $u$ of $\OO_K$ such that $u \pi_K^{{n}/{r}}=2$.
Now let $1+a\pi_D^i\in (1+\mathbf m_D^i)\setminus (1+\mathbf m_D^{i+1})$, with $a\in U_D$. 
Then $$(1+a\pi_D^i)^2=1+au\pi_D^{i+\frac{dn}{r}}+(a\pi_D^i)^2\in  1+\m_D^{i+\frac{dn}{r}}\, ,$$
since $i > \frac{dn}{r}$, and we easily deduce  from the isomorphism $(1+\m_D^i)/(1+\m_D^{i+1})\cong \OO_D/\m_D$ used
in the proof of (\ref{skewfield_third}) that
$$(1+\m_D^i)^2=(1+\m_D^{i+\frac{dn}{r}})\, .$$ In particular,
 $U_i^2=U_{i+\frac{dn}{r}}$.   

\end{proof}

\begin{Corollary} \label{dodd}Let $K$ be a finite extension of $\Q_2$, $D$ a $K$-central division algebra of dimension $d^2$ and  $\LL=\sli_1(D)$. Assume that the Sylow pro-2 subgroups of $\Aut_{\Q_2}(\LL)$ are $\R$-finite. Then $d$ is odd.  
\end{Corollary}

\begin{proof}
Since $\LL$ is simple, the Sylow pro-$2$ groups of $\Aut_{\Q_2}(\LL)$ are just infinite (see Proposition III.9 of \cite{KLP}). 
By Lemma \ref{propertiesRfinite}, $\R$-finite just infinite  pro-2 groups are torsion free,  so $\Aut_K(\LL)\subseteq \Aut_{\Q_2}(\LL)$ has no
$2$-torsion elements. Thus $d$ must be odd by Proposition \ref{autsl1new} and Proposition \ref{pgl1d}(1, 2).
\end{proof}
 
\begin{Corollary} \label{kle25} Let $K$ be a finite extension of $\Q_2$, $D$ a $K$-central division algebra of dimension $d^2>1$ and  $\LL=\sli_1(D)$. Assume  that the Sylow pro-2 subgroups of $\Aut_{\Q_2}(\LL)$  have at most  25 real irreducible characters. Then $d=3$ and $K=\Q_2$. 
\end{Corollary}

\begin{proof}
Let $q=2^r$ be the order of the residue field $\OO_K/\m_K$ and $n=|K:\Q_2|$. 
Let $\rho\in \Aut_{\Q_2}(\LL)$. By Proposition \ref{autsl1new}, we may regard $\rho$ as a $\Q_2$-automorphism of $D$. Note that $\OO_D$ consists of integral over $\Z_2$ elements of $D$ and so $\OO_D$ and $\m_D$ are $\rho$-invariant.  
Then the restriction of $\rho$ to $\m_D^i$ induces an automorphism on the quotient   $\m_D^i/\m_D^{i+1}$.  If we identify $\m_D^i/\m_D^{i+1}$ with $  \F_{q^d}$ ($\m_D^i/\m_D^{i+1}$ is canonically isomorphic to $\OO_D/\m_D \cong  \F_{q^d}$), then $\rho$ acts   on $\m_D^i/\m_D^{i+1}$  as an element from $\Gal(\F_{q^d}/\F_2)\ltimes \F_{q^d}^*$ acts on $\F_{q^d}$. Note that in this way we obtain a group homomorphism $\Psi_i$ from $\Aut_{\Q_2}(\LL)$ into $\Gal(\F_{q^d}/\F_2)\ltimes \F_{q^d}^*$.

Let $P$ be   a Sylow pro-$2$ subgroup  of $\Aut_{\Q_2}(\LL)$. Since $|\Gal(\F_{q^d}/\F_2)\ltimes \F_{q^d}^*|=rd(q^d-1)$ and   $d$ is odd,  by Corollary \ref{dodd}, we obtain that $|\Psi_i(P)|\le r$.  
It follows from this that   the orbits of the action of a Sylow pro-$2$ group of $\Aut_{\Q_2}(\LL)$ on $\m_D^i/\m_D^{i+1}$ have size at most $r$. 
 
 Next we consider the action of $P$ on the subquotient $U_{\frac{dn}{r}}/U_{\frac{2dn}{r}}$ of $ \PGL_1(D)\cong \Aut_{K}(\LL)$. For each
$dn/r\leq i\leq 2dn/r-1$, $P$ acts by conjugation on $U_i/U_{i+1}$. 
 Recall that  $w\pi_D^{i}+\m_D^{i+1}\mapsto (1+w\pi_D^{i})U_{i+1}$ defines a group homomorphism from $\m_D^i/\m_D^{i+1}$ onto $U_i/U_{i+1}$, and it is easy to check that this map respects the action of $P$.  
 It follows that the $P$-orbits in $U_i/U_{i+1}$ have size at most $r$ and so by Proposition \ref{pgl1d}(3),
$$r_P((U_i\setminus U_{i+1})/U_{i+1})\ge \left \{ \begin{array} {ll}
\frac{q^d-1}r& \textrm{if $i\not \equiv 0\mod d$}\\
\frac{q^{d-1}-1}r  & \textrm{if $i \equiv 0 \mod d$}\end{array}\right . .
$$ 

Thus, we obtain that
$$r_P(U_{\frac{dn}{r}}/U_{\frac{2dn}{r}})\geq1+\sum_{i=\frac{dn}{r}}^{\frac{2dn}{r}-1}r_P((U_i\setminus U_{i+1})/U_{\frac{2dn}{r}})\geq\frac {n}{r^2}((d-1)(2^{rd}-1)+2^{r(d-1)}-1)+1\, .$$
Since $r_{ P}(U_{\frac{dn}{r}}/U_{\frac{2dn}{r}})\le 25$ and $d$ is odd, an easy calculation yields to $d=3$ and $n=1$.

\end{proof}

The following theorem implies Theorem B.
\begin{Theorem} \label{d3} Let $D$ be a $\Q_2$-central division algebra of dimension 9 and $\LL=\sli_1(D)$. Then the  Sylow pro-2 subgroups
 of $\Aut_{\Q_2}(\LL)$  have  exactly   25 real irreducible characters. 
\end{Theorem}

\begin{proof}
 By Proposition \ref{autsl1new} and by the Skolem-Noether  theorem, $$\Aut_{\Q_2}(\LL)\cong \Aut_{\Q_2}(D)\cong \PGL_1(D) = D^*/Z(D^*).$$
When there is no possible confusion, 
we shall identify an element in the group  $\Aut_{\Q_2}(\LL)$ with a preimage of it in $D$.  

Using the same notation as at the beginning of the section, let $w$ be a primitive $7$th root
of unity, and $\pi_D$ an element such that $\pi_D^3=2\in\Z_2$ and $(\pi_D)^{-1}w\pi_D=w^2$,  so we 
can assume that 
$$
D=\Q_2[w, \pi_{D}]\, .
$$Recall that the valuation ring $\OO_D=\Z_2[w, \pi_D]$ has a unique maximal ideal $\m_D=\pi_D\OO_D$, 
and it is clear that $2\OO_D=\m_D^3$ and $\OO_D/\m_D\cong \F_8$. 
We know that the group $U_1$ defined above is a Sylow pro-$2$ subgroup of $\PGL_1(D)$. 

We start with some calculations inside the group $U_1.$

\begin{Claim}\label{cent1} Let $i, j\ge 1$ and $x\in U_i\setminus U_{i+1}$. Then $C_{U_j}(xU_{i+j+1})=C_{U_j}(x)U_{j+1}$.
\end{Claim}

\begin{proof} The inclusion $C_{U_j}(x)U_{j+1}\le C_{U_j}(xU_{i+j+1})$ is obvious. Let us show the converse inclusion.  Write $i=i_0+3i_1$ and $j=j_0+3j_1$ with  $0\le i_0,j_0\le 2$.
The proof is divided in 4 subcases depending on whether $i_0$ and $j_0$ are equal or not to $0$. We consider only the  subcase when $i_0,j_0\ne 0$, the other cases are proved similarly.

We write $ x\equiv 1+a\pi_D^{i}\pmod{ U_{i+1}}$ with $a\in \OO_D\setminus \m_D$. First observe that since $i_0,j_0\ne 0$, $|C_{U_j}(x)U_{j+1}/U_{j+1}|\ge 2$. Indeed, if $i_0=j_0$, then $1+a\pi_D^i2^{j_1-i_1}\in C_{U_j}(x)\setminus U_{j+1}$ and if $i_0\ne j_0$, then $1+(a\pi_D)^{2i}2^{(j-2i)/3}\in C_{U_j}(x)\setminus U_{j+1}$.

Thus, it is enough  to show that $|C_{U_j}(xU_{i+j+1})/U_{j+1}|=2$. Take $y\in C_{U_j}(xU_{i+j+1})$ with 
$y\equiv 1+b\pi_D^j \pmod{U_{j+1}}$ and
$b\in\OO_D\setminus \m_D$. Since $[x,y]\in U_{i+j+1}$ we obtain that $ab^{4^i}-ba^{4^j}\in   \m_D$. 
Hence 
\begin{equation}\label{eqb}
b^{4^{i}-1}\equiv a^{4^j-1} \pmod{ \m_D}.
\end{equation} Recall that  $\OO_D/ \m_D\cong \F_8$. Note that the map $\F_8^*\to \F_8^*$ that sends $\bar b$ to $\bar b^{2^{2i}-1}$ is bijection, since $i\not \equiv 0\pmod 3$. Hence there exists  only one class of $b$ modulo $  \m_D$   that satisfies (\ref{eqb}). This implies that  $|C_{U_j}(xU_{i+j+1})/U_{j+1}|=2$.
\end{proof}
As a corollary we obtain the following.
\begin{Claim}\label{cent2} Let $i, j\ge 1$, $n> i+j$ and $x\in U_i\setminus U_{i+1}$. Then $C_{U_j}(xU_{n})=C_{U_j}(x)U_{n-i}$.
\end{Claim}
\begin{proof} The proof is done by induction on $n$. The base of induction ($n=i+j+1$) is done in the previous claim.\end{proof}
 
By Proposition \ref{pgl1d}(5), $U_i/U_{i+3}$ is elementary abelian for $i\geq 4$. Thus all the conjugacy classes in this group are real. 
Moreover, the same result implies that these are all the conjugacy classes of elements of order
at most $2$ in $U_i/U_{i+3}$.
In the following claim we count their number.

\begin{Claim}\label{25} Let $i\equiv 1\pmod 3$ and assume $i\ge 4$. Then $r_{U_1}(U_i/U_{i+3})=25$.
\end{Claim}
\begin{proof}
Recall that $U_{i+2}/U_{i+3}$ is central in $U_1/U_{i+3}$.
 By Claim \ref{cent2}, if $x\in U_i\setminus U_{i+1}$, then  $C_{U_1}(xU_{i+3})=C_{U_1}(x)U_3$ and
  if $x\in U_{i+1}\setminus U_{i+2}$ then  $C_{U_1}(xU_{i+3})=C_{U_1}(x)U_2$. 
  Moreover,  using that $i\equiv 1\pmod 3$,    an argument along the lines of the proof of Claim \ref{cent1} yields that if   $x\in U_i\setminus U_{i+1}$ then
$|U_1:C_{U_1}(x)U_3|=2^4$, and if $x\in U_{i+1}\setminus U_{i+2}$ then $|U_1:C_{U_1}(x)U_2|=2^2$.

  By the previous paragraph, since $i\equiv 1\pmod 3$ we have that Proposition  \ref{pgl1d}(\ref{skewfield_third}) implies that   there are  $14=\frac{|(U_i\setminus U_{i+1})/U_{i+3}|}{2^4}$ $U_1$-conjugacy classes in $(U_i\setminus U_{i+1})/U_{i+3}$, $7=\frac{|(U_{i+1}\setminus U_{i+2})/U_{i+3}|}{2^2}$ $U_1$-conjugacy classes in $(U_{i+1}\setminus U_{i+2})/U_{i+3}$ and $4=|U_{i+2}/U_{i+3}|$ $U_1$-conjugacy classes in $U_{i+2}/U_{i+3}$.

\end{proof}

By the comments before Claim \ref{25}, in order to finish the proof we only 
need to show that for large $i$, $U_1/U_i$ does not contain
real elements of order $4$.
 
\begin{Claim} Let $i\ge 4$.
Let $x\in U_i$ be such that $xU_{i+4}$ is a real element of order $4$ in $U_1/U_{i+4}$. 
Then there exists $y\in U_3$ such that $(xU_{i+4})^y=(xU_{i+4})^{-1}$. 
\end{Claim}

\begin{proof}
Note that $x^2\in U_{i+3}$ by Proposition \ref{pgl1d}(5).
Now if $y\in U_1$ is such that 
$(xU_{i+4})^y=(xU_{i+4})^{-1}$, then $y\in C_{U_1}(xU_{i+3})$,
and thus
by Claim \ref{cent2} we can assume that $y\in U_3$.
\end{proof}
 
\begin{Claim}\label{order4} If $i\ge 6$, then there are no real elements of order $4$ in $U_1/U_{i+4}$.
\end{Claim}

\begin{proof}  Let $xU_{i+4}$ be a real element of order 4 in $U_1/U_{i+4}$. 
  Then $x\in U_j\setminus U_{j+1}$ for $i-2\leq j \leq  i$. Since
the following arguments do not depend on the value of $j$, for simplicity we may
  assume that  $x=1+a\pi_D^i\in U_i$ with $a\in \OO_D\setminus \m_D$. Again we distinguish two cases depending on whether $i\equiv 0 \pmod 3$ or not. We only consider the case $i\not \equiv 0 \pmod 3$, the other case is proven in a similar way.
 
By the previous claim there exists $y=1+2b\in U_3$ ($b\in  \OO_D\setminus \m_D$) such that $$1+ 2(ab^{4^{i}}-ab)\pi_D^i\equiv [x,y]\equiv x^2\equiv (1+2a\pi_D^i)\pmod {U_{i+4}}.$$
Hence
$ab^{4^{i}}-ab\equiv a \pmod { \m_D}$. Since $a$ is invertible in $\OO_D $, we obtain that there should exist $\bar b\in \F_8\cong\OO_D/ \m_D$ such that 
$\bar b^{2^{2i}}-\bar b=\bar 1$. But this is impossible, obtaining a contradiction.
\end{proof}

Now we are ready to finish the proof of the theorem. Let
$i$ be a positive integer
congruent to $1$ modulo $3$. 
It suffices to prove that for any large enough $i$, 
the finite $2$-group 
$U_1/U_i$
has $25$ real conjugacy classes, 
by Lemma \ref{Brauer} and Lemma \ref{propertiesRfinite}(2). 
This follows from  Claims \ref{25} and \ref{order4}.
\end{proof}

\section{Proof of Theorem A}

\setcounter{Claim}{0}
In this section we obtain  Theorem A as a consequence of two results. Firstly we show that if there are infinitely many finite 2-groups $G$  as in the statement of Theorem A, then  there is an infinite pro-$2$ group having  exactly $r$ real  irreducible characters. Secondly we show that such a pro-2 group does not exist for odd $r$  and $r\le 23$.

\begin{Theorem} Let $r$ be a natural number.  Assume that  there are infinitely many finite 2-groups $G$ with $r(G)=r.$ Then  there is an infinite pro-$2$ group having  exactly  $r$ real  irreducible characters 
\end{Theorem}

\begin{proof}Let us consider the following directed graph $\Gamma$. The set of vertices $V(\Gamma)$ of $\Gamma$ consists of  the isomorphism classes of finite 2-groups with exactly $r$ real irreducible conjugacy classes. There is an edge from $G_1$ to $G_2$ if and only if   there exists a normal subgroup $Z$ of order 2 in $G_2$ such that $G_2/Z\cong  G_1$. We say also that $G_1$ is a father of $G_2$ and $G_2$ is a son of $G_1$

\begin{Claim}
There is a number $C$, such that if $G\in V(\Gamma)$ and $|G|>C$, then $G$ has at least one father.
\end{Claim}

\begin{proof} Let $G\in V(\Gamma)$. By Lemma \ref{fgroupsrclases}(3), there exists  an $r$-bounded $k$ such that $\Irr_r(G/ G_k)=\Irr_r(G)$.  Note that the  order of $G_i/G_{i+1}$ is $r$-bounded for all $1\le i\le k-1$ by Lemma \ref{fgroupsrclases}(1), whence the order of $G/ G_k$ is also $r$-bounded (say    $ |G/G_k|\le C=C(r)$). Thus, if $|G|>C$ then $ G_k   \ne \{1\}$. Let  $Z\le  G_k  $ be a normal subgroup of $G$ of order 2. Then  it is clear that $G/Z$ is a father of $G$.
\end{proof}

For any $G\in V(\Gamma)$ consider the subgraph $\Gamma_G$ of $\Gamma$ consisting of $G$ and all  its  descendants (the vertices of $\Gamma$ that can be reached by a path from $G$).

\begin{Claim} Let $2^n\ge C$. Then there exists a group $G\in V(\Gamma)$ of order $2^n$, such that $\Gamma_G$ is infinite. Moreover, if $H\in V(\Gamma)$ and $\Gamma_H$ is infinite, then there exists a son $G$ of $H$ such that $\Gamma_G$ is infinite. 
\end{Claim}

\begin{proof} From the previous claim any group  in $V(\Gamma)$ of order at least $2^n$ lies in $\cup_{G\in V(\Gamma),|G|=2^n} \Gamma_G$. Thus if $\Gamma$ is infinite, then some $\Gamma_G$ is infinite. The second statement of the claim is proved in a similar way.
\end{proof}

Now we are ready to finish the proof of the theorem. By the previous claim we can construct inductively  $G_i\in V(\Gamma), i\in\N $, such that $\Gamma_{G_i}$ is infinite and $G_{i+1}$ is a son of $G_i$. Then the pro-2 group isomorphic to the inverse limit of $\{G_i\}_{ i\in\N}$ has exactly $r$ real irreducible characters.
\end{proof}

\begin{Theorem} Let $P$ be a pro-2 group having $r$ real  irreducible characters.  If $r$ is odd and $r\le 23$, then $P$ is finite.
\end{Theorem}
  
\begin{proof}  By way of contradiction, assume that there exists an infinite pro-2 group $P$ having an odd number of real irreducible characters   smaller that $25$. Then $P$ is of finite rank. Thus, we may assume also that $\dim \LL(P)$ is  as small as possible. If $P$ is not just infinite, by Corollary \ref{jiq}, any of  its just infinite quotients satisfies our hypothesis. Thus, assume also that $P$ is just infinite. 
Then $P$ is an open subgroup of $\Aut(\LL)$,  where
$\LL= \LL(P)$ is a  semi-simple finite dimensional Lie $\Q_2$-algebra all whose simple components are isomorphic. Also, 
by Corollary \ref{nsol} $P$ is non-solvable.
Note also that $P$ is  torsion free, because all torsion elements of $P$ are in $rad_f(P)$ (Lemma \ref{propertiesRfinite}(3)).  
 
If $P$ is   not a Sylow pro-2 subgroup of $\Aut(\LL)$, then  there exists $P<Q\le  \Aut(\LL)$ such that $|Q:P|=2$. 
By Theorem \ref{embedding}, there exists $a\in Q$ such that $C_P(a)$ has an odd number of real characters and $|\Irr_r(C_P(a)|\le |\Irr_r(P)|$. 
Observe that $a$ acts non-trivially on $\LL $ and so the dimension of $\LL(C_P(a))\cong C_{\LL}(a)$ is smaller  than the dimension of $\LL$. This contradicts to our choice of $P$.
   
On the other hand, if $P$ is   a Sylow pro-2 subgroup of $\Aut(\LL)$ then, by Theorem
\ref{sylow}, $\LL(P)\cong \sli_1(D)$ for some division $\Q_2$-algebra $D$. Corollary \ref{kle25} implies that $D$  has   dimension 9 over its center  $\Q_2$. But then Theorem \ref{d3}  states   that $|\Irr_r(P)|=25>23$, a contradiction.
  
\end{proof}

\section{Further comments}
In this section we describe several possible directions for further research. As we have  proved the ``minimal" pro-2 group with  odd number of conjugacy classes belong to the  Sylow subgroups of $\PGL_1(D)$ where $D$ has dimension at least 9 over its center. In particular, these groups have at least 3 generators. So we suggest the following problem.

\begin{Conjecture} Let $k$ be an odd number. Then there exists only a finite number of finite 2-groups generated by 2 elements with exactly $k$ real conjugacy classes.
\end{Conjecture} 

In a similar fashion to Lemma \ref{propertiesRfinite} we believe that the following holds.

\begin{Conjecture}
Let $G$ be a finite $2$-group. Then $\rk (G)$ is $q$-bounded, where $q$ is the number of rational characters of $G$. 
\end{Conjecture}

\bibliographystyle{amsplain}

\end{document}